\documentclass[12pt,reqno]{amsart}

\usepackage{enumerate} 
\usepackage[utf8]{inputenc} 
\usepackage[T1]{fontenc} 
\usepackage{amssymb,amsthm}
\usepackage{mathrsfs}
\usepackage{textcomp}
\usepackage{bm}
\usepackage{graphicx}
\usepackage{xcolor}

\newtheorem{thm}{Theorem}
\newtheorem{prop}[thm]{Proposition}
\newtheorem{lem}{Lemma}

\theoremstyle{definition}

\theoremstyle{remark}

\numberwithin{equation}{section}

\usepackage[font=footnotesize,tableposition=top,figureposition=bottom,skip=0pt]{caption}
\usepackage{hyperref}
\hypersetup{%
	linkcolor=blue,
	anchorcolor=black,
	citecolor=red,
	filecolor=magenta,
	menucolor=red,
	urlcolor=magenta,
	colorlinks=true,
	bookmarksnumbered=true,
	bookmarksopen=true,
	bookmarksopenlevel=0,
	breaklinks=true
}

\begin{document}

\title[Topological Wiener-Wintner ergodic theorem ]{Topological Wiener-Wintner ergodic theorem with polynomial weights}

 \author{Ai-hua Fan}
\address{LAMFA, UMR 7352 CNRS, University of Picardie, 33 rue Saint Leu,80039 Amiens, France}
\email{ai-hua.fan@u-picardie.fr}
\thanks{}
\maketitle






\begin{abstract} For a totally uniquely ergodic dynamical system, we prove a topological Wiener-Wintner ergodic theorem with polynomial weights under
the coincidence  of the quasi discrete spectrums of the system in both senses
of Abramov and of Hahn-Parry. The result applies to ergodic nilsystems. Fully oscillating sequences can then be constructed on nilmanifolds. 

\end{abstract}





\section{Introduction}
Let $(X, \mathcal{B}, \mu, T)$ be an ergodic measure-preserving dynamical system. The Wiener-Wintner ergodic theorem states that for any $f\in L^1(\mu)$ and for almost all $x\in X$, the limit
\begin{equation}\label{WWT}
     \lim_{N\to \infty}\frac{1}{N} \sum_{n=0}^{N-1} e^{2\pi i n \alpha} f(T^n x)
\end{equation} 
exists for all $\alpha \in \mathbb{R}$ \cite{WW}.  Since the work \cite{WW}, several different proofs have appeared \cite{Furstenberg1960, BL, L1990}.
Bourgain \cite{Bourgain1990} (see also \cite{Assani}) proved that the above limit is uniform on $\alpha$ and  the limit is zero if $f\in E_1(T)^\perp$ where $E_1(T)$ is the set of eigenfunctions.

Lesigne \cite{L1990,L1993} proved a generalized  Wiener-Wintner theorem which states that for almost all $x$  the limit
\begin{equation}\label{PWWT}
     \lim_{N\to \infty}\frac{1}{N} \sum_{n=0}^{N-1} e^{2\pi i P(n)} f(T^n x)
\end{equation}
exists for all real polynomials  $P$. Under the further total ergodicity, the necessary and sufficient condition on $f$ was found for the limit in (\ref{PWWT}) to be zero \cite{L1993}.
   The notion of  Abramov's quasi discrete spectrum is used to describe  that condition. See \cite{Abramov} for this spectrum theory, which
   finds its origin in Halmos and von Neumann \cite{HN}. Recall that for the above ergodic system, 
one  defines inductively the group of $k$-th (measurable) quasi eigenfunctions by
$$
     E_{k}(T) = \{f \in L^2(\mu):  |f|=1, Tf \cdot \overline{f} \in E_{k-1}(T)\}, \quad \forall k\ge 1
$$     
where $E_0(T)$ denotes the group of eigenvalues. 
Lesigne proved that if we assume that  $(X, \mathcal{B}, \mu, T)$ is totally ergodic, then $f \in E_k(T)^\perp$ if and only if  for a.e. $x\in X$, the limit (\ref{PWWT})
is equal to zero 
for all $P\in \mathbb{R}_k[t]$, where $\mathbb{R}_k[t]$ denotes the set of polynomials of degree at most $k$ with real coefficients. 
Later Frantzikinakis \cite{Frantz2006} proved that the limit in (\ref{PWWT}) is uniform in $P\in \mathbb{R}_k[t]$, answering a question of Lesigne \cite{L1993} (pp. 771)
and generalizing the result of Bourgain mentioned above.  There is a version of Wiener-Wintner ergodic theorem with nilsequences as weights obtained by Host-Kra \cite{HK2009} (see also \cite{EZ_K}).

Lesigne's result can be restated as follows. If $f \in E_k(T)^\perp$, the sequence $f(T^n x)$ is oscillating of order $k$ for a.e. $x$. 
Recall that a sequence  $(w_n)_{n \ge 0}$ of complex numbers  is defined to be {\em oscillating of order $d$} ($d \ge  1$) if for any real polynomial $P \in \mathbb{R}_d[t]$  we have
$$
    \lim_{N\to \infty} \frac{1}{N}\sum_{n=0}^{N-1} w_n e^{ 2\pi i P(n)} =0.
$$
A {\em fully oscillating sequence} is defined to be an oscillating sequence of all orders. These two notions of oscillation were introduced in \cite{F}. The notion of oscillation of order $1$ was defined in \cite{FJ}, in order to consider 
questions similar to Sarnak's conjecture (\cite{Sarnak, Sarnak2}). Namely, 
for a given sequence $(w_n)$, we would like to find those topological dynamical systems $(X, T)$ of zero entropy such that
\begin{equation} \label{SarnakConj}
   \lim_{N\to \infty} \frac{1}{N}\sum_{n=0}^{N-1} w_n f(T^n x) =0
\end{equation}
for any $f \in C(X)$ and any $x \in X$. 
Sarnak's conjecture states that the limit in (\ref{SarnakConj}) is zero for all systems of zero entropy when $(w_n)$  is the M\"{o}bius function.
Sarnak's conjecture is proved for different systems \cite{Bourgain13a,Bourgain13b,BSZ,DK2015, EALdlR14,EALdlR16,EKLR17,F,F2,FJ,FKPLM15,Green-Tao2008,GT2012,HLSY2017,HWZ2016,HWY2017,LS,MR,Veech,Wang2017} .

One motivation of the present work is to find topological dynamical systems  $(X, T)$
and continuous functions $f$ such that $(f(T^n x))$ is fully oscillating or oscillating of order $d$ for all $x\in X$ without exception. 
If $T$ is an affine dynamics of zero entropy on a compact abelian group, there
 is no such function different from zero which gives  fully oscillating sequences \cite{F2, Shi}. But as we shall see,  we can find such functions
for some nilsystems, like  ergodic nilsystems on Heisenberg homogeneous spaces. \medskip

  There is already  a topological version of Wiener-Wintner
theorem due to Robinson \cite{Robinson} (see also Assani \cite{Assani}, Theorem 2.10). Let $(X, T)$ be a uniquely ergodic topological dynamical system. Suppose that $E_0(T)=G_0(T)$ where $E_0(T)$ (resp. $G_0(T)$) is the group of measurable (resp. continuous)  eigenvalues.
Then for any $f\in C(X)$ and any $x \in X$, the limit (\ref{WWT}) exists. Furthermore, the limit is zero if $f\in E_1(T)^\perp$. The condition $E_0(T)=G_0(T)$ is necessary to some extent. In fact,
Robinson constructed some strictly ergodic skew product on torus $\mathbb{T}^2$ such that  $E_0(T)\setminus G_0(T)\not=\emptyset$, for which there exist $e^{2\pi i \alpha} \in E_0(T)\setminus G_0(T)$,
$f\in C(X)$ and $x\in X$ such that the limit (\ref{WWT})  fails to exist.  See also \cite{Lenz2009}.

  We shall prove a topological version of Lesigne's Wiener-Wintner theorem, which generalizes to some extent Robinson's theorem. The condition we find will involve the quasi discrete spectrum of the system in the sense of Abramov \cite{Abramov} as well as  
  the quasi discrete spectrum of the system in the sense of Hahn-Parry \cite{HP0}. Recall that for a transitive topological dynamical system $(X,T)$,
  one  defines inductively the group of $k$-th (continuous) quasi eigenfunctions by
$$
     G_{k}(T) = \{f \in C(X):  |f|=1, Tf \cdot \overline{f} \in G_{k-1}(T)\}, \quad \forall k\ge 1.
$$     
  
  The main result in this paper is the following. 
  \medskip
  
{\bf Theorem A.}
{\em Let $(X, T)$ be a topological dynamical system and let $k\ge 1$ be an integer.  Suppose
\\
\indent {\rm  (H1)}  $(X, T^j)$ for $1\le j < \infty$ are all strictly ergodic. \\ 
\indent {\rm  (H2)}  $E_j(T) = G_j(T)$ for all $0\le j \le k$. \\
 Then for any continuous function $f \in C(X)$, the following assertions are  equivalent:\\
\indent \mbox{\rm (a) }  $f \in G_k(T)^\perp$;\\
\indent \mbox{\rm (b) }   for any  $x\in X$, we have
 \begin{equation}\label{lim_k}
  \lim_{N\to \infty }\sup_{P \in \mathbb{R}_k[t]} \left| \frac{1}{N} \sum_{n=0}^{N-1} e^{2\pi i P(n)} f(T^n x)\right|  =0.
\end{equation}
}

 If a system satisfies  (H1), we say it is {\em totally uniquely ergodic}.  The condition (H2) is referred to as the {\em coincidence of spectrums}
 up to order $k$.  
 
 Some result similar to Theorem A was obtained by Eisner and Zorin-Kranch \cite{EZ_K} where the sequence $e^{2\pi i P(n) \alpha}$ is replaced by nilsequences produced by Sobolev functions, but it was assumed that 
 the projection $f$ to some Host-Kra factor is zero, and consequently $f$ is orthogonal to certain Host-Kra factor (Host-Kra factor being introduced in \cite{HK2005}), not only to the Abramov factor. However it was only assumed
 in \cite{EZ_K} that $(X, T)$ is uniquely ergodic.  
 
 An application of the main theorem to ergodic nilsystems leads to the following theorem.

\medskip

{\bf Theorem B.}  
{\em 
Let $G$ be a connected and simply connected nilpotent Lie group, $\Gamma$ a discrete cocompact subgroup of $G$ and $g\in G$. Let $X=G/\Gamma$ be the nilmanifold and let   $T: X \to X$ be defined by $x\Gamma \mapsto gx\Gamma$.
Suppose that $(X, T)$ is uniquely ergodic. Then for any $F \in C(X)$ such that $F\in G_\infty(T)^\perp$ and any $x \in G$, the sequence $F(g^n x \Gamma)$ is fully oscillating. 
}
\medskip

Applied to Heisenberg groups, Theorem B gives us the following result, which was mentioned in \cite{F3}.
\medskip

{\bf Theorem C.}  
{\em  Let $m \ge 1$ and let $\alpha_1, \cdots, \alpha_m; \beta_1, \cdots, \beta_m$ be real numbers. Suppose  $1, \alpha_1, \cdots, \alpha_m, \beta_1, \cdots, \beta_m$
are $\mathbb{Q}$-independent.   For any continuous  function $\varphi \in C(\mathbb{T})$ such that $\int \varphi(x) dx=0$, the sequence 
$$
n\mapsto \varphi (n \alpha_1[n\beta_1] + \cdots + n \alpha_m[n\beta_m] )$$
 is fully oscillating.
}
\medskip

The function $ n\mapsto \varphi (n \alpha_1[n\beta_1] + \cdots + n \alpha_m[n\beta_m] )$ is a special generalized polynomial. Generalized polynomials, especial their uniform distributions,  have been
well studied by Haland \cite{Haland1993,Haland1994,Haland1995}, Bergelson and Leibman \cite{BL2007},  Leibman \cite{Leibman1998,Leibman2012}. Notice that 
the sequence $(e^{2\pi i n^{d+1}\alpha})$  with $\alpha$ irrational is uniformly distributed on $\mathbb{S}^1$, oscillating of order $d$ but not oscillating of order $d+1$.  
The oscillation is a notion relative to but different  from the uniform distribution. 

\medskip
Here is the sketch of the proof of Theorem A, which is a long argument by induction on $k$. The main idea is inspired by Lesigne \cite{L1993}. The proof of the case $k=1$ is essentially a simple application of Van der Corput 
inequality and Krylov-Bogoliubov theorem, but the proof of the uniformity on $\alpha$ (see  Theorem \ref{order1} (4)) follows an idea of Frantzkinakis \cite{Frantz2006}
(this idea is also used  in the proof of Proposition \ref{PropF}).  The Van der Corput inequality
also allows us to reduce the order $k$ of the polynomial $P$  to a polynomial of order $k-1$, by induction  (see Theorem \ref{orderk}). But  in this way the result is only proved  for all polynomials  but some exceptions.  To deal with these
exceptional polynomials,  we convert the problem to that of some unique ergodic extension of $(X, T^j)$ in the sense of Furstenberg \cite{Furstenberg1961} (see Lemma \ref{Ext-Erg2}).


\medskip

We make preparations in Section 2 (recall of two notions of quasi-discrete spectrums)  and Section 3 (extension of unique ergodicity) in order to prove Theorem A in Section 4. Theorem B and Theorem C are proved in Section 5. 

\section{Quasi discrete spectrums}

We recall here the two theories of spectrum, one measure-preserving and the other topological . 

\subsection{Definitions of two quasi-discrete spectra}
   
    Let $ (X,T) $ be a topological dynamical system. Assume that $ (X,T) $ is transitive, i.e. 
    the orbit $ O(x) := \{ T^n x : n \ge 0 \}$ of some $ x \in X $ is dense in X. Let $ C(X)$ be the 
    Banach algebra of continuous complex valued functions on $ X $ and $G(X)$ be the subset of $C(X)$ consisting of $f$ such that $|f(x)|=1$ for all $x\in X$. It is clear that $G(X)$ is a group 
    with multiplication as group operation. The quasi-discrete spectrum concerns the 
    isometry $f \mapsto f\circ T$  on $ C(X) $, which is still denoted by $T$, namely $ T f = f\circ T$. Now let us recall the 
    notion of quasi-discrete spectrum of Hahn-Parry \cite{HP0}, a notion similar to    Abramov's  on measure-theoretic 
    dynamics \cite{Abramov} which uses the concept of quasi-eigenfunction due to Halmos and von Neumann \cite{HN}. 
    
    We say that $ f \in C(X)$, $f \ne 0$, is an \textit{eigenfunction} if there is a complex number $\lambda \in \mathbb{C} $ 
    for which
    \begin{equation}\label{eq:1-4}
    f\circ T= \lambda f.
    \end{equation}
    The number $ \lambda $ is called an \textit{eigenvalue}. Let $ H_1 $ be the set of all eigenvalues. 
    The eigenfunctions corresponding to the eigenvalue $1$ are called {\em invariant functions}. 
    The transitivity of 
    $T$ implies that invariant functions are constant functions, and $H_1 \subset \mathbb{S}$ where $ \mathbb{S}$ is the group
    $\{z \in \mathbb{C} : |z| = 1\}$ under multiplication, and eigenfunctions have constant modulus. Denote 
    \[
    G_1:= \{f \in G(X) : \exists \lambda \in \mathbb{C} \ \mbox{\rm such \ that}\   Tf=\lambda f \}. 
    \]
    It is the group of eigenfunctions. 
    Let $G_0=H_1$ and let us identify a constant with a constant function.
   Thus we have $H_1 =  G_0 \subset G_1$.
   
    Quasi-eigenvalues and quasi-eigenfunctions have different orders. They are inductively defined. 
    Assume that  subgroups $H_n$ and $G_n$ of $G(X)$ are defined in a such a way that
     $$ H_1 \subset H_2 \subset \cdots \subset H_n; \quad
     G_1 \subset G_2 \subset \cdots \subset G_n; \forall i<n,  H_{i+1} \subset G_i.$$
    We define  $G_{n+1} $ to be the set of all $ f_{n+1} \in G(X)$ 
    such that there is a $ g_n \in  G_n $ with 
    \begin{equation}\label{eq:1-5}
    f_{n+1}\circ T = g_n f_{n+1}.
    \end{equation}
Then we define $H_{n+1}$ to be the set of all $g_n \in G_n $ for which there is an $f_{n+1} \in  G_{n+1}$ satisfying   
\eqref{eq:1-5}. Let
\[
G_\infty := \bigcup_{n=1}^{\infty}G_{n}, \quad H_\infty := \bigcup_{n=1}^{\infty}H_{n}.
\]
The elements in the group $H_\infty$ are called \textit{quasi-eigenvalues} and the elements in the group $G_\infty$ are called \textit{quasi-eigenfunctions}.
For $n \ge 2$, the elements in $H_n \setminus H_{n-1}$ are called \textit{$n$-th quasi-eigenvalues} and the elements in $G_n\setminus G_{n-1}$ are called \textit{$n$-th quasi-eigenfunctions}.
If necessary, we shall write $G_n(T)$ and $G_\infty(T)$ for $G_n$ and $G_\infty$. The notations $H_n(T)$ and $H_\infty(T)$ are sometimes also useful.   

A dynamical system $(X, T)$ is said to have \textit{quasi-discrete spectrum} if the algebra generated
by the quasi-eigenfunctions is dense in $C(X)$, or equivalently the linear span of quasi-eigenvalues
is dense in $C(X)$ because $G$ is a multiplicative group. By using Stone-Weierstrass
theorem we see that this is equivalent to say that quasi-eigenfunctions separate points of
$ X $. If, furthermore, $ G_d = G_{d+1} $ and $d_T$ is the least such integer, we say that $(X,T)$ has \textit{quasi-discrete
spectrum of order} $d_T$.

\subsection{Orthogonality of quasi-eigenfunctions}

\begin{prop} Let $(X, \mathcal{B}, \mu , T)$ be an ergodic measure-preserving  dynamical system. Suppose that there is no eigenvalue of finite order
(except the eigenvalue $1$). Then all quasi-eigenfunctions are orthogonal. 
\end{prop}

\begin{proof} 
Let $E(T)$ be the group of all $f\in L^\infty(\mu)$ such that $|f(x)|=1$ a.e.. 
Let us first make a remark: suppose 
\begin{eqnarray*}
      Tf_2 & = & f_1 f_2, \ \ Tf_1   = h f_1; \\
      Tg_2 &=& g_1 g_2, \ \ Tg_1 = h g_1 
\end{eqnarray*}
where $f_1, f_2, g_1, g_2, h \in E(T)$, 
then we have $f_1 = c g_1$ for some eigenvalue $c$. In fact, first observe that $g_1/f_1$. By the ergodicity we  get $g_1 = c f_1$ for some $c\in \mathbb{S}^1$.
Then from $Tf_2  =  f_1 f_2$ and $Tg_2  =  c f_1 g_2$, we get
$$
    \frac{Tg_2}{Tf_2} = c \frac{g_2}{f_2}. 
$$
So, $c$ is an eigenvalue to which the eigenfunction $g_2/f_2$ is associated.

Let us consider two arbitrary different quasi-eigenfunctions $f$ and $g$ which are not proportional. We are going to show $\int f \overline{g} d\mu =0$.
More precisely, let $f$ be a quasi-eigenfunction of order $k$ and $g$  be a quasi-eigenfunction of order $\ell$. Assume $1\le k \le \ell <\infty$.
In other words, 
$$
   Tf = f_{k-1} f, \ \ Tf_{k-1} = f_{k-2}f_{k-1}, \ \ \cdots, \ \ Tf_1 = f_0 f_1
$$
$$
   Tg = g_{\ell-1} g, \ \ Tg_{\ell-1} = g_{\ell-2}g_{\ell-1},\ \  \cdots, \ \ Tg_1 = g_0 g_1
$$
for some $f_j\in E(X)$ ($1\le j<k-1$) and  $g_j\in E(X)$ ($1\le j<\ell-1$), where $f_0$ and $g_0$ are two eigenvalues. 
By an inductive argument, we deduce that 
$$
    f(T^n x) = f(x) f_{k-1}(x)^{\binom{n}{1}} f_{k-2}(x)^{\binom{n}{2}} \cdots f_{1}(x)^{\binom{n}{k-1}} f_{0}(x)^{\binom{n}{k}}.
$$
Let $f_j(x) =e^{2\pi i \theta_j}$ $(0\le j \le k-1)$, where $\theta_j$ ($1\le j<k$) depends on $x$, but $\theta_0$ doesn't. We get 
$$
    f(T^n x) = f(x) e^{2\pi i p_x(n)}, \quad \mbox{\rm with}\ \  p_x(n) =  \theta_0 \binom{n}{k} + \theta_1 \binom{n}{k-1}+\cdots + \theta_{k-1} \binom{n}{1} .
$$
Similarly we have 
$$
    g(T^n x) = g(x) e^{2\pi i q_x(n)}, \quad \mbox{\rm with}\ \  q_x(n) =  \phi_0 \binom{n}{\ell} + \phi_1 \binom{n}{\ell-1}+\cdots + \phi_{\ell-1} \binom{n}{1} 
$$
where $\phi_j \in [0,1)$ is the argument of $g_j(x) = e^{2\pi i \phi_j}$. By the invariance we get 
$$
   \int f\overline{g} d\mu =   \int f\overline{g}  e^{2\pi i \big( p_x(n) - q_x(n) \big)}d\mu
$$
which holds for all $n$, so that we have
\begin{eqnarray*}
   \int f\overline{g} d\mu & = &   \lim_{N\to \infty}\frac{1}{N}\sum_{n=0}^{N-1} \int f(x)\overline{g(x)}  e^{2\pi i \big( p_x(n) - q_x(n) \big)}d\mu(x)\\
   & = &  \int f(x)\overline{g(x)}  \lim_{N\to \infty}\frac{1}{N}\sum_{n=0}^{N-1}  e^{2\pi i \big( p_x(n) - q_x(n) \big)}d\mu(x)
\end{eqnarray*}
when the last limit exists for all $x$ (for the last equality we use the Lebesgue dominated convergence theorem). This limit does exist and is equal to zero. Thus we finish the proof. 

We prove that  the limit is equal to zero by induction on $\ell$.  Assume $\ell=1$. Then $k=1$ and $f$ and $g$ are eigenfunctions. 
If $f_0\not=g_0$, it is well known that $f$ and $g$ are orthogonal. The case $f_0=g_0$ is not possible, because otherwise, $f$ and $g$ are proportional
as eigenfunctions associated to the same eigenvalue.  

Now we suppose the conclusion holds for $\ell-1$. 
We prove the case $\ell\ge 2$ by distinguishing several cases. 


{\em Case I. $k=\ell$, $f_0\not=g_0$}:  In this case, $p_x - q_x$ is a real polynomial of degree $\ell$ with leading coefficient $ (\theta_0-\varphi_0)/\ell !$ which is irrational, because 
$\theta_0-\varphi_0 \not =0$ is the argument of the eigenvalues $f_0 \overline{g}_0$. Without use of the induction hypothesis we conclude by Weyl theorem, because $p_x(n) -q_x(n)$ is uniformly distributed.

{\em Case II. $k=\ell$, $f_0 =g_0$}: In this case, we first apply the above remark to get $g_1 = c f_1$ where $c=e^{2\pi i \xi}$ is an eigenvalue.
 If $c\not= 1$, then $\xi$ is irrational and 
 $$
         p_x(n) - q_x(n) = \xi \binom{n}{k-1} + (\theta_2-\phi_2)\binom{n}{k-2}+\cdots
 $$
 is a polynomial of degree $k-1$ with leading irrational coefficient $ \xi \binom{n}{k-1}$. We conclude with Weyl theorem. If $c=1$, we get $g_1=f_1$ and fall into the following situation
 \begin{eqnarray*}
       Tf_3  =  f_2 f_3, \ \ Tf_2 = f_1f_2, \ \ (Tf_1=f_0 f_1)\\
        Tg_3  =  g_2 g_3, \ \ Tg_2 = f_1 g_2, \ \  (Tf_1=f_0 f_1).
 \end{eqnarray*}
 Again we apply the above remark to get $g_2= d f_2$ for some eigenvalue $d$.  If $d\not= 1$, we conclude. Otherwise we get $g_2=f_2$. 
 In this inductive  way, we can conclude otherwise we get  $g=f$, a contradiction.   
 
 {\em Case III. $k<\ell$}:  If $g_0=1$, then $g_1$ is an invariant function then constant. Thus we can forget the trivial equality $Tg_1=g_0g_1$ and just start with $Tg_2= g_1g_2$.
 In other words, we have reduced $\ell$ to $\ell -1$. Therefore  we can apply the induction hypothesis to conclude. If $g_0\not= 1$, then
  $p_x - q_x$ is a real polynomial of degree $\ell$ with leading coefficient $-\varphi_0/\ell !$ which is irrational. We conclude by 
 Weyl theorem.
\end{proof}

\subsection{Quasi-discrete spectra of $T$ and of $T^m$}

The following lemma has a version for measure-preserving dynamical systems, due to Lesigne \cite{L1993}, for which the ergodicity and the total ergodicity replace the transitivity and the total minimality.

\begin{lem} \label{EigenInv} Let $T_1$ and $T_2$ be two transitive maps on a compact metric space $X$. Suppose that $T_1$ and $T_2$ are commutative.
Then $G_k(T_1) = G_k(T_2)$ for all $k\ge 1$. In particular, if both $T$ and $T^m$ are transitive ($m\ge 2$) , then 
$$
         G_k(T) = G_k(T^m)
$$
for all $k\ge 1$.
\end{lem}
\begin{proof} The proof is the same as in \cite{L1993}. It suffices to replace the ergodicity by the transitivity which ensures that a continuous invariant function is constant.
 We include the proof here for completeness. 
We prove it by induction on $k$. 
We assume $T_1f = \lambda f$ with $f \in G_1(T_1)$ and $\lambda \in \mathbb{S}$. By the commutativity, we have
$$
    T_1 (T_2f \cdot \overline{f}) = T_2 T_1 f \cdot  T_1 \overline{f} =|\lambda|^2 T_2  f \cdot   \overline{f}
    = T_2  f \cdot   \overline{f}.
$$
By the transitivity of $T_1$, we deduce that the $T_1$-invariant function $T_2  f \cdot   \overline{f}$ is a constant, so $f \in G_1(T_2)$. Thus $G_1(T_1)\subset G_1(T_2)$. By the symmetry, we get 
$G_1(T_1) = G_1(T_2)$.

Now assume that $G_j(T_1) = G_j(T_2)$ for all $1\le j<k$ ($k\ge 2$).  Let $f \in G_k(T_1)$. Then there exists $g \in G_{k-1}(T_1)$ such that 
$$T_1f = gf.$$
By the induction hypothesis, $g \in G_{k-1}(T_2)$. Then there exists $h \in  G_{k-2}(T_2)$ such that $$T_2g= h g.$$
 Thus, by the commutativity, we have
$$
    T_1(T_2 f) = T_2(T_1f) = T_2(g f) = T_2g \cdot T_2f = hg T_2f.
$$
It follows that
$$
        T_1(T_2 f \cdot \overline{f}) = hg T_2f \cdot \overline{g f}  = h (T_2f \cdot \overline{ f}).
$$
Since $h \in  G_{k-2}(T_2)$, we have $h \in  G_{k-2}(T_1)$ by the induction hypothesis. Therefore
$$
     T_2 f \cdot \overline{f} \in E_{k-1}(T_1).
$$
Again, by the induction hypothesis,
$$
     T_2 f \cdot \overline{f} \in G_{k-1}(T_2).
$$
So,  $f \in G_k(T_2)$. 
By the symmetry,  we have $G_k(T_1) = G_k(T_2)$.
\end{proof}

\section{Extension of unique ergodicity}
Assume that $(X, T)$ is a uniquely ergodic topological dynamical system with $\mu$ as the only invariant measure.  Let  $G$ be a compact abelian group with normalized Haar measure $m$
and $\phi: X \to G$ be a continuous map. Define the map $S:=S_\phi:  X\times G \to X\times G$ by
\begin{equation}\label{def-S}
       S(x, z) = (Tx, \phi(x) z)).
\end{equation}
The dynamical system $(X\times G, S)$ is called a group extension of $(X, T)$. The product measure $\mu \times m$ is $S$-invariant.
The following lemma of Furstenberg \cite{Furstenberg1961} (p. 579) gives the condition for a group extension  to be still uniquely ergodic
(the "only if" part is obvious).

  \begin{lem} [\cite{Furstenberg1961}] \label{lem_Furs}Suppose that $(X, T)$  is unique ergodic with $\mu$ as its invariant measure.  Then 
  the above defined  extension $(X\times G, S)$ is uniquely ergodic if (and only if) the $S$-invariant measure $\mu \times m$ is ergodic.
  It is the case iff the following equation
  $$
       \phi(x)^k = \frac{h(Tx)}{h(x)}
  $$
  has no solution for $k\not =0$ an integer and $h$ a measurable function.
  \end{lem}

 Let $(X, T)$ be a topological dynamical system and let $p\ge 1$ be an integer. Suppose $$
 \phi_1: X \to \mathbb{S}^1,\ \ \  \phi_2: X \times   \mathbb{S}^1 \to \mathbb{S}^1,  \ \ \cdots, \ \ 
 \phi_p: X \times   \mathbb{S}^{p-1} \to \mathbb{S}^1$$ are  given continuous maps. Then $S_1: X\times \mathbb{S} \to X\times \mathbb{S}$
 defined by
 $$
        S_1(x, z_1) = (Tx, \phi_1(x) z_1)
 $$
 is a group extension of $(X, T)$, and $S_2: X\times \mathbb{S}^2 \to X\times \mathbb{S}^2$
 defined by
 $$
        S_2(x, z_1, z_2) = (Tx, \phi_1(x) z_1, \phi(x, z_1)z_2)
 $$
  is a group extension of $(X\times \mathbb{S}^1, S_1)$. Inductively, we define $S_3, \cdots, S_p$ such that 
  $S_{j+1}$ is a group extension of $S_j$. In particular, the map
  $S_p: X\times \mathbb{S}^p \to X\times \mathbb{S}^p$ is
 defined by
 $$
        S_p(x, z_1, \cdots, z_p) = (Tx, \phi_1(x) z_1, \phi(x, z_1)z_2, \cdots, \phi_{p}(x, z_1, \cdots, z_{p-1}) z_p).
 $$
 We call it  a $p$-th group extension of $(X, T)$.
 
 Let us consider the following  special case
 $$
     \phi_1(x)= \gamma(x), \ \ \phi(x, z_1) = z_1, \ \cdots,  \  \phi_p(x, z_1, \cdots, z_{p-1})=z_{p-1}
 $$
 where $\gamma : X \to \mathbb{S}^1$ is a given continuous map.

 Let us consider the following special case where $G=\mathbb{S}^p$ ($p\ge 1$) and 
 \begin{equation}\label{phip}
      S(x, z_1, \cdots, z_p) = (Tx, \gamma(x) z_1,  z_1z_2, \cdots,  z_{p-1} z_p).
 \end{equation}
 with $\gamma : X \to \mathbb{S}$ a continuous map. 
 
 For any given eigenfunction $\tilde{\gamma} \in G_1(T)$, we will find a number $\lambda\in \mathbb{S}^1$
 such that the $p$-th extension $S$ defined (\ref{phip})  with $\gamma=\lambda \tilde{\gamma}$ have the following property: whenever
 $T^n$ is uniquely ergodic, so is $S^n$. 
 By Lemma \ref{lem_Furs}, it suffices to deduce the ergodicity of $\mu\times m$ relative to $S^n$ from the ergodicity of $\mu$ relative to $T^n$.
 
 
 
  \begin{lem} \label{Ext-Erg2} Let $(X, T)$ be uniquely ergodic with invariant measure $\mu$ and let $p\ge 1$  be integer.
  For any eigenfunction $\widetilde{\gamma} \in G_1(T)$, there exists a number $\lambda \in \mathbb{S}$  such that
  the extension $S$ on $X \times \mathbb{S}^p$ of $(X, T)$ 
 defined by (\ref{phip}) with  $\gamma:= \lambda\widetilde{\gamma} \in G_1(T)$ has the following properties:\\
 \indent \mbox{\rm (1)} $S$ is uniquely ergodic. \\
  \indent \mbox{\rm (2)} for any integer $n\ge 1$, $S^n$ is uniquely ergodic if $T^n$ is uniquely ergodic.
 \end{lem}

\begin{proof}   (1) The proof of this part is contained in \cite{L1993} (pp. 779-780) and we repeat it here for completeness.  We first discuss the following cocycle
equation (\ref{cobord2}), which is also useful for part (2). In \cite{L1993}, only the case $k=0$ was discussed. The general case with  arbitrary $k$ would have been discussed \cite{L1993},
because the powers of  the extension were used.

Let $\xi\in H_1(T)$ be the eigenvalue associated to the eigenfunction $\widetilde{\gamma}$. 
 We introduce a parameter $\lambda \in \mathbb{S}$, to be determined later,  and consider the equation 
 \begin{equation}\label{cobord2}
        \xi^k (\lambda \widetilde{\gamma}(x))^j = \frac{h(Tx)}{h(x)}     \  \ \ \mu\!-\!a.e.
 \end{equation}
 where the unknown is the triple $(k, j, h)$ with $k\in \mathbb{Z}$, $j \in \mathbb{Z}\setminus\{0\}$ and $h: X\to \mathbb{S}$ a Borel function.
 We claim that there exists $\lambda\in \mathbb{S}$ such that (\ref{cobord2}) has no solution. Since $\mathbb{S}$ is uncountable, it suffices to show that for any fixed couple $(k, j)\in \mathbb{Z}
 \times \mathbb{Z}\setminus \{0\}$, there are at most countably many $(\lambda, h)$ such that (\ref{cobord2}) is solvable.   That is really the case.
In fact,
if $(\lambda_1, h_1)$ and $(\lambda_2, h_2)$ are distinct solutions, then
$$
     \lambda_1^j h_1(x) \overline{h_2(x)} = \lambda_2^j h_1(Tx) \overline{h_2(Tx)}.
$$
If furthermore $\lambda_1^j \not=\lambda_2^j$,  then $h_1$ and $h_2$ are orthogonal.  But  any family of orthogonal functions is countable because $L^2(\mu)$ is separable.  
So, there are at most countable many possibilities $\lambda^j$.  To finish the argument, we just remark that $\lambda_1^j =\lambda_2^j$ means $\lambda_2 = \lambda_1e^{2\pi i l/j}$ ($0\le l <j$).

In the following we fix a number $\lambda\in \mathbb{S}$ such that  (\ref{cobord2}) has no solution. Notice that this $\lambda$ depends on $\widetilde{\gamma}$.
Let $\gamma:=\lambda \widetilde{\gamma} \in G_1(T)$.
  Consider  the extension $S$ defined by
  $$  S(x, z_1, z_2, \cdots, z_p)
      = (Tx, \gamma(x)z_1, z_1z_2, \cdots, z_{p-1}z_p) . 
      $$
      
      According to Lemma \ref{lem_Furs}, in order to prove that $S$ is uniquely ergodic, it suffices to prove that $\mu\times m$ is $S$-ergodic.    
Suppose that $f$ is a bounded $S$-invariant function. By a Fourier method, we are going to prove that $f$ is constant. 
For $J:=(j_1, \cdots, j_p) \in \mathbb{Z}^p$, define
$$
      f_J(x) = \int_{\mathbb{S}^p} f(x, z_1, \cdots, z_p) z_1^{j_1}\cdots z_p^{j_p} dz_1 \cdots dz_p.
$$
It is a Fourier coefficient of the function $z \mapsto f(x, z)$.
We make the change of variables $(Z_1, \cdots, Z_p) := (\gamma(x)z_1, z_1z_2, \cdots, z_{p-1}z_p)$, which preserves the Haar measure $dz$, to get
\begin{eqnarray*}
   & & f_J(Tx) = \int_{\mathbb{S}^p} f(Tx, Z_1, \cdots, Z_p) Z_1^{j_1}\cdots Z_p^{j_p} dZ_1 \cdots dZ_p\\
    &=&  \int_{\mathbb{S}^p} f(Tx, \gamma(x) z_1, z_1z_2, \cdots, z_{p-1}z_p)  \gamma(x)^{j_1} z_1^{j_1+j_2}z_2^{j_2+j_3}\cdots z_{p-1}^{j_{p-1}+j_p} z_p^{j_p} dz_1 \cdots dz_p.
\end{eqnarray*}
Thus
\begin{equation}\label{RR}
   f_{j_1, \cdots, j_p}(Tx)   = \gamma(x)^{j_1} f_{j_1+j_2, \cdots, j_{p-1}+j_p, j_p} (x).
\end{equation}
Then, by the $T$-invariance of $\mu$, we get
\begin{equation}\label{Recurve}
        \int_X  | f_{j_1, j_2 \cdots, j_p, j_p} (x)|^2 d\mu(x)  = \int_X  | f_{j_1+j_2, \cdots, j_{p-1}+j_p, j_p} (x)|^2 d\mu(x). 
\end{equation}
On the other hand, notice that
\begin{eqnarray*}
    \sum_{J\in \mathbb{Z}^p} \int_X  | f_{J} (x)|^2 d\mu(x)
    = \int_{X}\int_{\mathbb{S}^p} |f(x, z)|^2 d\mu(x) dz <\infty.
\end{eqnarray*}
It follows that 
\begin{equation}\label{lim_0}
   \lim_{|J|\to \infty} \int_X  | f_{j_1, j_2, \cdots,  j_p} (x)|^2 d\mu(x) =0.
\end{equation}
From (\ref{Recurve}) and (\ref{lim_0}), we deduce that for $\mu$-almost every $x$, $f_J(x)=0$ when at least one of $j_2, \cdots, j_p$ is non zero.  
That is to say,  $f$ depends only on $x$ and $z_1$. 
Write $f_j(x) = f_{j_1, 0, \cdots, 0}(x)$. Then (\ref{RR}) becomes 
\begin{equation}\label{sol}
f_j(Tx) = \gamma(x)^j f_j(x).
\end{equation} So, $|f_j|$ is $T$-invariant by the ergodicity. We claim that $|f_j|=0$. Otherwise, we get a contradiction to the non-solvability of
of the equation (\ref{cobord2}).   Therefore $f$ depends only on $x$. But $\mu$ is $T$-ergodic, so $f$ is almost everywhere constant.
Thus we have proved the $S$-ergodicity of $\mu\times m$, then the unique ergodicity of $S$. 
  
    (2) Recall $T\gamma = \xi \gamma$. We have the following formula
      for the powers of $S$:
      \begin{equation}\label{Sn}
        S^n(x, z_1, \cdots, z_p) = (T^nx, Z_1, \cdots, Z_p)
      \end{equation}
      where
      \begin{eqnarray*}
         Z_1 &=& 
          \xi^{\binom{n}{2}} \gamma(x)^{\binom{n}{1}} z_1;\\
           Z_2 &=& \xi^{\binom{n}{3}} \gamma(x)^{\binom{n}{2}} z_1^{\binom{n}{1}} z_2;\\ 
                   & \vdots &\\
              Z_p &=& \xi^{\binom{n}{p+1}} \gamma(x)^{\binom{n}{p}} z_1^{\binom{n}{p-1}} \cdots z_{p-1}^{\binom{n}{1}} z_p.       
      \end{eqnarray*}
 We can prove the formula (\ref{Sn}) by induction on $n$ using the Pascal formula. Here $\binom{n}{k}=0$ for $k>n$.
 Notice that $S^n$ is a $p$-th extension of $T^n$.
 We have assumed that $T^n$ is uniquely ergodic.
 Again, according to Lemma \ref{lem_Furs}, in order to prove that $S^n$ is uniquely ergodic, we have only to prove that 
 $\mu\times m$ is $S^n$-ergodic.   
 Suppose that $f$ is a bounded $S^n$-invariant function.  
For $J:=(j_1, \cdots, j_p) \in \mathbb{Z}^p$, we also define
$$
      f_J(x) = \int_{\mathbb{S}^p} f(x, z_1, \cdots, z_p) z_1^{j_1}\cdots z_p^{j_p} dz_1 \cdots dz_p.
$$
 Consider $(z_1, \cdots, z_p) \mapsto (Z_1, \cdots, Z_p)$ as a change of variable, which preserves the Haar measure $dz$. We have
\begin{eqnarray*}
   & & f_J(T^nx) = \int_{\mathbb{S}^p} f(T^nx,  Z_1, \cdots, Z_p) Z_1^{j_1}\cdots Z_p^{j_p} dZ_1 \cdots dZ_p\\
    &=&  \int_{\mathbb{S}^p} f(S^n(x,x, \gamma(x) z))  \xi^a \gamma(x)^{b} z_1^{c_1}z_2^{c_2}\cdots  z_p^{c_p} dz_1 \cdots dz_p
\end{eqnarray*}
where 
\begin{eqnarray*}\label{abc}
    a & =& j_1\binom{n}{2} +  j_2\binom{n}{3} +\cdots + j_p\binom{n}{p+1}\\\
    b & = & j_1\binom{n}{1} +  j_2\binom{n}{2}+ \cdots + j_p\binom{n}{p}\\
    c_1& = & j_1\binom{n}{0} +  j_2\binom{n}{1} +\cdots + j_p\binom{n}{p-1}\\
    c_2& = & j_2\binom{n}{0} + j_3 \binom{n}{1}+ \cdots + j_{p}\binom{n}{p-2}\\
         & \vdots & \\
         c_{p-1}& = & j_{p-1}\binom{n}{0} +  j_p\binom{n}{1}\\
         c_p&=& j_p.
\end{eqnarray*}
Thus we get a formula generalizing (\ref{RR})
\begin{equation}\label{RR2}
   f_{j_1, \cdots, j_p}(T^nx)   = \xi^a \gamma(x)^{b} f_{c_1, \cdots, c_p} (x).
\end{equation}
As in the proof of (1), from (\ref{RR2}) we can deduce that $f$ only depends on $x$ and $z_1$.   Let $f_j(x) = f_{j, 0, \cdots, 0}(x)$.
Then (\ref{RR2}) becomes 
\begin{equation}\label{RR2b}
   f_{j}(T^nx)   = \xi^{j n(n-1)/2} \gamma(x)^j f_{j} (x).
\end{equation}
The  non solvability of (\ref{cobord2}) implies that $f_j(x)=0$ for $j \not=0$. Thus $f$ depends only on $x$.  
The $S^n$-invariance of $f$ implies the $T^n$-invariance of $f$. Finally we conclude that $f$ is constant by the $T^n$-ergodicity of $\mu$.
 \end{proof}
 
 Notice that the above choice $\lambda$ is valid for all $n\ge 1$.
 \medskip
 
 The quasi eigenfunctions of the extension $S$ is simply related to those of the base dynamics $T$, as the following lemma shows.

 \begin{lem} [\cite{L1993}, p.782] \label{Eigen}
 Let $S$ be the extension defined by $\gamma \in G_1(T)$. Let $1\le k\le p$.  If $g \in G_k(S)$, there exists $\widetilde{g} \in G_{k+1}(T)$
 and $d_1, \cdots,  d_j \in \mathbb{Z}$ such that
 $$
     g(x, z_1, \cdots, z_p) = \widetilde{g}(x) z_1^{d_1}\cdots z_k^{d_k}.
 $$
 \end{lem}

\section{TWWT-Proof of Theorem A}
The first two results below concerning  Topological Wiener-Wintner Theorem (TWWT) constitute the first two steps towards  the proof by induction of Theorem A. 
For their proofs, we don't need the results in Section 2 and Section 3. The proofs given here are adapted from Lesigne \cite{L1993} who treated measure-preserving systems instead of topological
systems.  Another argument used in the proof of Theorem A is due to Frantzikinakis \cite{Frantz2006} (see Proposition \ref{PropF}).
 
 \subsection{Orthogonality to polynomials of degree $1$ }

The following Theorem \ref{order1}  is mainly due to Assani \cite{Assani1993} (see also \cite{Assani}, p. 42). 
A particular case of Robinson's theorem \cite{Robinson} asserts that the limit in (\ref{lim_2}) is uniform on $x$ for fixed $\alpha$.
As pointed out in \cite{Robinson}, B. Weiss obtained some unpublished similar results. 
We  first give a direct of the pointwise convergence of (\ref{lim_2}),  based on the Krylov-Bogolioubov theorem and
the inequality of Van der Corput and Herglotz theorem (through  the spectral measure). Then we prove (\ref{lim_3}) as a
 Robinson's uniform consequence mentioned above and of a technique due to Frantzikinakis \cite{Frantz2006}. This technique of Frantzikinakis
will be used once more in a more involved way in the proof of Proposition \ref{PropF} where polynomials, in stead of $n\alpha$, are concerned.

\begin{thm}\label{order1}
Let $(X,T)$ be a uniquely ergodic system with the unique ergodic measure $\mu$. Suppose that $E_0(T) = G_0(T)$. Let $f \in C(X)$ and  $\alpha \in [0, 1)$.\\
\indent \mbox{\rm (1)} If $e^{2\pi i \alpha } \in E_0(T)$, we have
\begin{equation}\label{lim_1}
  \forall x \in X, \lim_{N\to \infty } \frac{1}{N}\sum_{n=0}^{N-1} e^{2\pi i n \alpha} f(T^n x) = g(x)^{{-1}}\int f g d\mu
\end{equation}
where $g$ is an eigenfunction associated to $e^{2\pi i \alpha}$ (unique up to multiplicative constant).\\
\indent \mbox{\rm (2)} If $e^{2\pi i \alpha } \not\in E_0(T)$, the limit in (\ref{lim_1}) is zero.\\
\indent  \mbox{\rm (3)} We have $f \in E_1(T)^\perp$ if and only if 
\begin{equation}\label{lim_2}
 \forall \alpha \in [0,1),  \forall x \in X, \lim_{N\to \infty } \frac{1}{N}\sum_{n=0}^{N-1} e^{2\pi i n \alpha} f(T^n x) = 0.
\end{equation}
\indent  \mbox{\rm (4)} 
We have $f \in E_1(T)^\perp$ if and only if 
\begin{equation}\label{lim_3}
  \lim_{N\to \infty } \sup_{\alpha \in \mathbb{R}}\left\|\frac{1}{N}\sum_{n=0}^{N-1} e^{2\pi i n \alpha} f(T^n \cdot)\right\|_{C(X)} = 0.
\end{equation}

\end{thm}

\begin{proof}  (1) Let $\lambda=e^{2\pi i \alpha}$.  Assume $g(Tx) = \lambda g(x)$. Then by the unique ergodicity, the limit in (\ref{lim_1}) is uniform on $x$ and is  equal to 
$$
    \lim_{N\to \infty } \frac{1}{N}\sum_{n=0}^{N-1} g(x)^{-1}g(T^n x) f(T^n x) = g(x)^{-1} \int f g d\mu.
$$
Thus we have checked (1). 

(2) We follow Lesigne \cite{L1993}  by using the inequality of Van der Corput. In the present topological case, Krilov-Bogoliubov theorem
will be used in the place of Birkhoff ergodic theorem.  First remark that the unique ergodicity
implies that for any $x\in X$ and any $h \in \mathbb{N}$, we have
\begin{equation}\label{lim-h}
 \lim_{N\to \infty } \frac{1}{N}\sum_{n=0}^{N-1}  f(T^{n+h} x) \overline{f(T^n x)} = \int f\circ T^h\cdot  \overline{f} d\mu = \int_{\mathbb{T}} e^{2\pi i h t} d\sigma(t)
 \end{equation}
 where $\sigma$ is the spectral measure associated to $f$. Let $0\le H <N$. By the inequality of Van der Corput, we have
 \begin{eqnarray*}
     & & \left| \frac{1}{N}\sum_{n=0}^{N-1} e^{2\pi i n \alpha} f(T^n x) \right|^2\\
     & \le &  \frac{N+H}{N(H+1)}\cdot \frac{1}{N} \sum_{n=0}^{N-1}  |f(T^n x)|^2 
      + 2 \frac{N+H}{N(H+1)^2} \\
    & & \ \ \ \  \times \left| \sum_{h=1}^{H} \frac{(H+1-h)(N-h)}{N} \cdot e^{2\pi i h \alpha} \cdot \frac{1}{N-h} \sum_{n=0}^{N-h-1} f(T^{n+h} x) \overline{f(T^n x)}\right| 
 \end{eqnarray*}
 Then taking limit as $N$ tends to infinity leads to
 \begin{eqnarray*}
    \limsup_{N\to \infty} \left| \frac{1}{N}\sum_{n=0}^{N-1} e^{2\pi i n \alpha} f(T^n x) \right|^2\
      &\le &   \frac{\langle f, f\rangle}{H+1} \\
      & & \hspace{-6em} +
      \left| \int_{\mathbb{T}}\left[  \frac{2}{(H+1)^2} \sum_{h=1}^{H} (H+1-h) \cdot e^{2\pi i h (\alpha +t)} \right] d \sigma(t) \right|. 
 \end{eqnarray*}
 Here we have used (\ref{lim-h}).
Since 
$$
    \lim_{H\to \infty} \frac{1}{H}\sum_{h=1}^H e^{2\pi i (\alpha +t)} = 0 \ \ \ \mbox{\rm if} \ \ \alpha + t \not\in \mathbb{Z},
$$
as the second order C\'esaro mean we have 
\begin{equation}\label{lim-h2}
    \lim_{H\to \infty} \frac{1}{(H+1)^2}\sum_{h=1}^H (H+1 -h) e^{2\pi i (\alpha +t)} = 0 \ \ \ \mbox{\rm if} \ \ \alpha + t \not\in \mathbb{Z}.
\end{equation}
On the other hand, since $e^{2\pi i \alpha} \not\in E_0(T)$, we have  $e^{-2\pi i \alpha} \not\in E_0(T)$ too. So, the measure $\sigma$ have no measure at $t = -\alpha$
and the limit in (\ref{lim-h2}) is $\sigma$-almost everywhere equal to $0$. Finally we can conclude for (2) by using the dominated convergence theorem of Lebesgue.

(3) is a direct consequence of (1) and (2).

(4) Because of (3), we have only to prove the uniform convergence (\ref{lim_3}) for $f \in E_1(T)^\perp$. 
Suppose that for any $\alpha \in [0,1]$
we have
\begin{equation}\label{F1}
         \lim_{M\to \infty} \limsup_{N\to \infty} \frac{1}{N} \sum_{n=1}^N \left\| \frac{1}{M}\sum_{m=1}^M e^{2\pi i m \alpha} b_{Mn +m}\right\|_{\mathbb{B}}=0.
\end{equation} 
Then
\begin{equation}\label{F2}
          \lim_{N\to \infty} \sup_{\alpha \in [0,1]} \frac{1}{N} \left\| \sum_{n=1}^N e^{2\pi i n \alpha} b_{n}\right\|_{\mathbb{B}}=0.
\end{equation}   
This is Lemma 2.2 in \cite{Frantz2006}, where $b_n$'s are complex numbers. But the proof is identical when $b_n$'s are in a Banach space $\mathbb{B}$.
Now notice that 
$$
   \left\|\frac{1}{M}\sum_{m=1}^M  e^{2\pi i m \alpha} f(T^{Mn +m} x) \right\|_{C(X)}\le  \left\|\frac{1}{M}\sum_{m=1}^M  e^{2\pi i m \alpha} f(T^{m} x)  \right\|_{C(X)}.
$$
The RHS in the above inequality tends to zero by Robinson's theorem. Thus the condition in (\ref{F1}) is satisfied by 
$b_n = f\circ T^n$ in the Banach space $(C(X), \| \cdot\|_{C(X)})$. Then we get $(\ref{F2})$ with $b_n = f\circ T^n$ and $\mathbb{B}=C(X)$.
This is what we have to prove.
\end{proof}

Theorem \ref{order1} asserts the existence of the limit  for every $x\in X$, which is even uniform on $x$, but under 
the imposed condition $E_0(T) = G_0(T)$. This condition  cannot been dropped. In fact, Robinson showed that there is a strictly ergodic analytic Anzai skew product  $T$ on the torus $\mathbb{T}^2$,
which has an essentially discontinuous eigenvalue (i.e. $E_0(T) \setminus G_0(T)\not= \emptyset$), and for such an eigenvalue $e^{2\pi i \alpha}$
and for some $f\in C(\mathbb{T}^2)$ the limit in (\ref{lim_1}) fails to exist for some point $x\in \mathbb{T}^2$ (Proposition 3.1 in \cite{Robinson}). 

 \subsection{Orthogonality to polynomials of degree $k$ }
Let 
$$D(T) :=\{ \alpha \in [0, 1): \exists m \in \mathbb{Z}\setminus \{0\}, e^{2\pi i m \alpha} \in E_0(T)\}.$$
The set $D(T)$ represents the roots of eigenvalues of $T$.

\begin{thm}\label{orderk}
Let $(X,T)$ be a uniquely ergodic system with the unique ergodic measure $\mu$. Suppose that $E_0(T) = G_0(T)$. Let $f \in C(X)$ and  $\alpha \in [0, 1)$.
 If $\alpha  \not\in D(T)$, then for $k \ge 1$
 \begin{equation}\label{lim_k}
  \forall x \in X, \ \ \ \lim_{N\to \infty } \sup_{P \in \mathbb{R}_{k-1}[t]}\left|\frac{1}{N}\sum_{n=0}^{N-1} e^{2\pi i (n^k \alpha + P(n))} f(T^n x) \right|= 0.
\end{equation}
\end{thm}

\begin{proof} Following again Lesigne \cite{L1993} (pp.773-774), we prove it by induction on $k$ using again the inequality of Van der Corput.  The case $k=1$ is (2) of Theorem \ref{order1} (3).
Assume the conclusion for $k\ge 1$. This hypothesis of induction applied to $f\circ T^h \cdot \overline{f}$ gives us
\begin{equation}\label{hyper}
   \lim_{N\to \infty } \sup_{Q \in \mathbb{R}_{k-1}[t]}\left|\frac{1}{N}\sum_{n=0}^{N-1} e^{2\pi i (n^k \alpha + Q(n))} f(T^{n+h} x)\overline{f(T^n x)} \right|= 0.
\end{equation}
The unique ergodicity implies that 
\begin{equation}\label{h2}
   \lim_{N\to \infty } \frac{1}{N}\sum_{n=0}^{N-1} |f(T^n x)|^2 = \langle f, f\rangle.
\end{equation}
Let $0\le H <N$. By the inequality of Van der Corput, we have
 \begin{eqnarray*}
     & & \left| \frac{1}{N}\sum_{n=0}^{N-1} e^{2\pi i (n^{k+1} \alpha +P(n))} f(T^n x) \right|^2\\
     & \le &  \frac{N+H}{N(H+1)}\cdot \frac{1}{N} \sum_{n=0}^{N-1}  |f(T^n x)|^2 
      + 2 \frac{N+H}{N^2(H+1)^2} \sum_{h=1}^H (H+1-h) \\
    & & \ \ \ \  \times \left|   \sum_{n=0}^{N-h-1}  e^{2\pi i ([(n+h)^{k+1} -n^{k+1}] \alpha + [P(n+h)-P(n)])} f(T^{n+h} x) \overline{f(T^n x)}\right|.
 \end{eqnarray*}
 Notice that $d^{\circ} (P(\cdot +h)-P(\cdot))\le k-1$ and 
 $(n+h)^{k+1} - n^{k+1} =(k+1)n^k + R$ with $d^{\circ} R \le k-1$.
 It follows that
  \begin{eqnarray*}
     & & \sup_{P \in \mathbb{R}_k[t]}\left| \frac{1}{N}\sum_{n=0}^{N-1} e^{2\pi i (n^{k+1} \alpha +P(n))} f(T^n x) \right|^2\\
     & \le &  \frac{N+H}{N(H+1)}\cdot \frac{1}{N} \sum_{n=0}^{N-1}  |f(T^n x)|^2 
      + 2 \frac{N+H}{N^2(H+1)^2} \sum_{h=1}^H (H+1-h) \\
    & & \ \ \ \  \times  \sup_{Q \in \mathbb{R}_{k-1}[t]} \left|   \sum_{n=0}^{N-h-1}  e^{2\pi i ((k+1)hn^k \alpha + Q(n))} f(T^{n+h} x) \overline{f(T^n x)}\right| 
 \end{eqnarray*}
 Since $e^{2\pi i(k+1)h\alpha} \not\in E_0(T)$, by (\ref{hyper}) and (\ref{h2}) we get
 $$
    \limsup_{N\to \infty} \sup_{P \in \mathbb{R}_k[t]}\left| \frac{1}{N}\sum_{n=0}^{N-1} e^{2\pi i (n^{k+1} \alpha +P(n))} f(T^n x) \right|^2
    \le \frac{\langle f, f \rangle}{H+1}.
 $$
 Letting $H\to \infty$ finishes the proof by induction.
\end{proof}

\subsection{Frantzikinakis lemma}
The following proposition is another ingredient for proving our Theorem A. It allows us to prove the uniformity on $P \in \mathbb{R}_k[t]$ of the convergence. 
There is a version for totally ergodic measure-preserving systems due to Frantzikinakis \cite{Frantz2006}, and the convergence for individual polynomials is due to Lesigne \cite{L1993}. For the proof for totally uniquely ergodic topological systems, we mimick \cite{Frantz2006}. 
Actually the proof  for our topological systems is 
simpler.  

\begin{prop} \label{PropF}Let $(X, T)$ be a totally uniquely ergodic topological dynamical system and let $f\in C(X)$.
Suppose that for any $\alpha\in \mathbb{R}$ and any $x\in X$ we have
\begin{equation}\label{FF1}
      \lim_{N\to \infty} \sup_{P\in \mathbb{R}_{k-1}[t]} \left|\frac{1}{N} \sum_{n=0}^{N-1} e^{2\pi i (n^k \alpha +P(n))} f(T^n x)\right|=0.
\end{equation} 
Then for any $x$ we have 
\begin{equation}\label{FF2}
      \lim_{N\to \infty} \sup_{P\in \mathbb{R}_{k}[t]} \left|\frac{1}{N} \sum_{n=0}^{N-1} e^{2\pi i P(n)} f(T^n x)\right|=0.
\end{equation} 
\end{prop}
 \begin{proof}
     We first claim that for any $\alpha\in \mathbb{R}$ and any $x\in X$ we have
     \begin{equation}\label{FF3}
   \lim_{M\to \infty}    \lim_{N\to \infty} \sum_{n=1}^{N}\frac{1}{N} \sup_{P\in \mathbb{R}_{k-1}[t]}  \left| \frac{1}{M}\sum_{m=0}^{M-1} e^{2\pi i (m^k \alpha +P(m))} f(T^{Mn+m} x)\right|=0.
\end{equation} 
In fact, let
$$
    g_M(\alpha, x) = \sup_{P\in \mathbb{R}_{k-1}[t]} F_M(\alpha, P, x)
    $$
    where
    $$
     F_M(\alpha, P, x) :=  \left| \frac{1}{M}\sum_{m=0}^{M-1} e^{2\pi i (m^k \alpha +P(m))} f(T^{Mn+m} x)\right|.
$$
The function $F_M$ depends only on the fractional parts of the coefficients of $P$. So, each $P\in  \mathbb{R}_{k-1}[t]$ can be identified   as a point in $\mathbb{R}^k/\mathbb{Z}^k$.
The function $F_M$ depends only on the fractal part of $\alpha$ either.
Thus $F_M$ is a continuous function of $(\alpha, P, x) \in \mathbb{R}/\mathbb{Z}\times  \mathbb{R}^k/\mathbb{Z}^k\times X$. It follows that
$g_M $ is a continuous function of $\alpha$ and $x$.   By applying Krilov-Bogoliubov theorem to the system $(X, T^M)$ and to the function $g_M(\alpha, \cdot)$, 
we get
\begin{eqnarray*}
 \lim_{N\to \infty} \sum_{n=1}^{N}\frac{1}{N} \sup_{P\in \mathbb{R}_{k-1}[t]}  \left| \frac{1}{M}\sum_{m=0}^{M-1} e^{2\pi i (m^k \alpha +P(m))} f(T^{Mn+m} x)\right|
 =\int g_M(\alpha, x) d\mu(x)
\end{eqnarray*}
where $\mu$ is the unique invariant measure. The hypothesis (\ref{FF1}) means that $g_M(\alpha, x)$ converges to zero for every point $x$.
This and Lebesgue's bounded convergence theorem allow us to conclude for (\ref{FF3}) if we take limit as $M\to \infty$.

 Fix $\alpha$ and $x$. Secondly we claim that for any $\epsilon >0$, there exist an integer $N_\alpha$ and an open neighborhood $V_\alpha$
 of $\alpha$ (both $N_\alpha$ and $V_\alpha$ depending on $\epsilon$ and $x$ too) such that
 \begin{equation}\label{FF4}
     \forall N >N_\alpha, \quad \sup_{\beta \in N_\alpha} 
      \sup_{P\in \mathbb{R}_{k-1}[t]}  \left| \frac{1}{N}\sum_{n=0}^{N-1} e^{2\pi i (n^k \beta +P(n))} f(T^n x)\right|\le \epsilon.
\end{equation} 
In fact, consider the general term $e^{2\pi i (n^k \beta +P(n))} f(T^n x)$ as a function of $\alpha$ and we denote it by $a_n(\alpha)$. 
For any $\beta$, we have
\begin{eqnarray*}
     \frac{1}{N} \sum_{n=0}^{N-1} a_n(\beta)
    & = & \frac{1}{[N/M]}    \sum_{n=1}^{[N/M]}   \frac{1}{M} \sum_{m=0}^{M-1} a_{nM +m}(\beta) + O(M/N).
    \end{eqnarray*}
where the constant involved in $O(\cdot)$ is $2\|f\|_\infty$. Observe that   $$(M n +m)^k\beta = m^k\beta + P_{M, n, \beta}(m)$$ with $P_{M, n, \beta}\in \mathbb{R}_{k-1}[t]$. 
For any $\beta$ and $\alpha$, we can write
 $$
     e^{2\pi i (Mn+m)^k\beta} = \Big(  e^{2\pi i m^k \beta} -  e^{2\pi i m^k \alpha}\Big) e^{2\pi i P_{M, n, \beta}(m)} 
     + e^{2\pi i (m^k \alpha+ P_{M, n, \beta}(m))}
 $$
 Thus
 \begin{equation}\label{FF5}
     \left|\frac{1}{N} \sum_{n=0}^{N-1} a_n(\beta)\right|
    \le  A_{M, N} (\alpha, \beta)+ B_{M,N}(\alpha, \beta)  + O(M/N)
\end{equation}
where 
\begin{eqnarray*}
    A_{M, N}(\alpha, \beta) & =& \frac{1}{[N/M]}    \sum_{n=1}^{[N/M]}   \frac{1}{M} \sum_{m=0}^{M-1} \left|e^{2\pi i m^k \beta} -  e^{2\pi i m^k \alpha}\right|  |f(T^{Mn+m} x)|\\
    B_{M, N}(\alpha, \beta) &=& \frac{1}{[N/M]}    \left|\sum_{n=1}^{[N/M]}   \frac{1}{M} \sum_{m=0}^{M-1} e^{2\pi i (m^k \alpha+ P_{M, n, \beta}(m) +P(Mn +m))} f(T^{Mn +m} x) \right|.
\end{eqnarray*}
 Now fix $\alpha$. By our first claim (see (\ref{FF3})), there exists an integer $M_\alpha$ such that for $N$ large enough and for all $\beta$ we have
 $$
      |B_{M_\alpha, N}(\alpha, \beta)|\le \frac{1}{[N/M_\alpha]}    \sum_{n=1}^{[N/M_\alpha]}  \sup_{Q\in \mathbb{R}_{k-1}[t]}  \left|\frac{1}{M_\alpha} \sum_{m=0}^{M_\alpha-1} e^{2\pi i (m^k \alpha+ Q(m))} f(T^{M_\alpha n +m} x) 
      \right|\le \frac{\epsilon}{3}.
 $$
 Now we deal with $A_{M_\alpha, N}$.  Choose a neighborhood $V_\alpha$ of $\alpha$ such that
 $$
        \sup_{\beta\in V_\alpha} \sup_{1\le m \le M_\alpha} \left|e^{2\pi i m^k \beta} -  e^{2\pi i m^k \alpha}\right| \le \frac{\epsilon}{3\|f\|}.
 $$
 Then for all $N$ and all $\beta \in V_\alpha$ we have
 $$
       A_{M, N}(\alpha, \beta)\le \frac{\epsilon}{3}.
 $$
 For $N$ large enough we have $2\|f\|M_\alpha/N\le \frac{\epsilon}{3}$. Thus we have proved  (\ref{FF4}). 
 
 We conclude for (\ref{FF2}) from $(\ref{FF4})$ by using a finite covering argument for the compact set $[0,1]$ where $\alpha$ varies. 
 
  \end{proof}

\subsection{TWWT with polynomial weights}

We restate Theorem A as follows.

\begin{thm}\label{OTWWT} Let $(X, T)$ be a uniquely ergodic topological dynamical system and let $k\ge 1$ be an integer.  Suppose
that the invariant measure has $X$ as support and
\\
\indent {\rm  (H1)}  $(X, T^j)$ for $1\le j < \infty$ are all  uniquely ergodic. \\ 
\indent {\rm  (H2)}  $E_j(T) = G_j(T)$ for all all $0\le j \le k$. \\
 For any continuous function $f \in C(T)$, the following assertions are equivalent\\
\indent \mbox{\rm (a) }  $f \in G_k(T)^\perp$;\\
\indent \mbox{\rm (b) }   for  $x\in X$, we have
 \begin{equation}\label{lim_k}
  \lim_{N\to \infty }\sup_{P \in \mathbb{R}_k[t]} \left| \frac{1}{N} \sum_{n=0}^{N-1} e^{2\pi i P(n)} f(T^n x)\right|  =0.
\end{equation}
\end{thm}

\begin{proof} $(b)$ {\it implies} $(a)$:  Let $g \in G_k(T)$. Then there are $g_j \in G_j(T)$ with $g_k = g$ such that
$$
    g_j(T x) = g_{j-1}(x) g_j(x)  \ \ (1\le j\le k). 
$$ Then
$
     g(T^n x) = e^{2\pi i P(n)}
$ with
$$
    P(t) =\sum_{j=0}^k \theta_j \binom{t}{k-j} \in \mathbb{R}_k[t] 
$$
where $e^{2\pi i \theta_j} = g_j(x)$. Therefore, by the Krylov-Bogoliubov theorem and (\ref{lim_k}), we get 
$$
    \int gf d\mu = \lim_{N\to \infty }\frac{1}{N}\sum_{n=0}^{N-1} g(T^n x) f(T^n x)  =0.
$$

 $(a)$ {\it implies} $(b)$:  We prove that (H1), (H2) and (a) imply (b),  by induction on $k$. The case $k=1$ was already proved (see Theorem \ref{order1} (3)). Let $k\ge 2$ and assume that the result   is true for $k-1$.  
 We are going to prove that  (H1), (H2), (a) and the induction hypothesis imply (b). By Theorem \ref{orderk} and Proposition \ref{PropF},
 it suffices to prove that for $\alpha \in D(T)$ and any $f\in G_k(T)^\perp \cap C(X)$ we have
 \begin{equation}\label{AA}
  \forall x \in X, \ \ \ 
      \lim_{N\to \infty }  \sup_{ Q\in \mathbb{R}_{k-1}[t]} \left|\frac{1}{N}\sum_{n=0}^{N-1} e^{2\pi i (n^k \alpha +Q(n))} f(T^n x)  \right|=0.
 \end{equation}

 That $\alpha \in D(T)$ means $\eta: =e^{2\pi i \ell \alpha} \in G_0(T)$ for some integer   $\ell \in \mathbb{Z}\setminus \{0\}$.  Notice that $\xi:=\eta^{k!} \in G_0(T)$. Let   $\widetilde{\gamma} \in G_1(T)$
 be an eigenfunction associated to $\xi$, i.e. 
 $$
     T\widetilde{\gamma}  =  \xi \widetilde{\gamma}.
 $$ 
Let us consider an extension  $(X\times \mathbb{S}^{k-1}, S)$ of $(X, T)$, as that in Lemma \ref{Ext-Erg2} with $p=k-1$,  where 
 $$
     S(x; z_1, \cdots, z_{k-1}) = (Tx; , \gamma(x)z_1, z_1z_2, \cdots, z_{k-2}z_{k-1}),   \quad \gamma = \lambda \widetilde{\gamma}.
 $$
 By Lemma \ref{Ext-Erg2}, we can choose $\lambda\in \mathbb{S}$ such that $S^j$  ($j\ge 1$) are all uniquely ergodic because  $T^j$ ($j\ge 1$) are all assumed uniquely ergodic. That is to say $S$ verifies (H1).
 It is clear that the invariant measure of $S$ has full support.
  By lemma \ref{Eigen}, $S$ verifies (H2) with $k$ replaced by $k-1$, because $T$ verifies (H2) with $k$. Then we can apply the induction hypothesis to $S$ to obtain:
 for any $F\in G_{k-1}(S)^\perp$, we have
 \begin{equation}\label{middle}
     \forall \omega \in X\times \mathbb{S}^{k-1},  \ \ \ \lim_{N\to \infty} \sup_{ P\in \mathbb{R}_{k-1}[t]} \left| \frac{1}{N}\sum_{n=0}^{N-1} e^{2\pi i P(n)} F(S^n \omega) \right|= 0. 
 \end{equation} 
 Choose $F(x, z_1, \cdots, z_{k-1})= f(x) z_{k-1}$. For $\omega =(x, 1,\cdots, 1)$, we have 
 $$
      F(S^n \omega) = f(T^n x) \xi^{\binom{n}{k}} \gamma(x)^{\binom{n}{k-1}}
 $$
 where  we have used the formula (\ref{Sn}) for the expression of $S^n$.
 Since $f\in G_{k}(T)^\perp$, we have $F\in G_{k-1}(S)^\perp$ by Lemma \ref{Eigen}.
 So, we can apply (\ref{middle}) to the function $F(x, z_1, \cdots, z_{k-1})= f(x) z_{k-1}$ and the point $\omega=(x, 1,\cdots, 1)$. This gives
 \begin{equation*}
     \forall x \in X, \ \ \  \lim_{N\to \infty}
  \sup_{Q\in \mathbb{R}_{k-1}[t]} \left| \frac{1}{N}\sum_{n=0}^{N-1} e^{2\pi iQ(n) }\xi^{\binom{n}{k}} f(T^n x)\right| = 0.
 \end{equation*}
 Here we have used the facts that  $|\gamma(x)|=1$ and $\binom{n}{k-1}$ is a polynomial of degree $k-1$. Now observe that
 $$
    \xi^{\binom{n}{k}} = \eta^{k! \binom{n}{k}}= e^{2\pi i n(n-1)\cdots (n-k+1)\ell \alpha} = e^{2\pi i (n^k \ell \alpha + R(n))} 
 $$
  with $R\in \mathbb{R}_{k-1}[t]$. Thus we can conclude that if $e^{2\pi i \ell \alpha} \in G_0(T) $ we have
  \begin{equation}\label{Conl-2}
    \forall x \in X, \ \ \  \lim_{N\to \infty} \sup_{Q\in \mathbb{R}_{k-1}[t]} \left| \frac{1}{N}\sum_{n=0}^{N-1} e^{2\pi i(Q(n) +  n^k \ell \alpha)} f(T^n x) \right|= 0.
 \end{equation}
 
 Now we are going to take off $\ell$ in (\ref{Conl-2}) in order to finish the proof.
 Since $\eta = e^{2\pi i \ell \alpha}\in G_0(T)$, we have $\eta^\ell \in G_0(T^\ell)$. Then $e^{2\pi i \ell^k \alpha} \in G_0(T^\ell)$ because $G_0(T^\ell)$ is a group.
 We are going to apply ( \ref{Conl-2}) to the system $(X, T^\ell)$. First remark that, 
 by the transitivity of $T^j$ and Lemma \ref{EigenInv}, the system $(X, T^\ell)$ has the properties (H1) and (H2).
 On the other hand, as $f\in G_k(T)^\perp$ and  $G_k(T)$ is stable under $T$,  we have $f\circ T^j \in G_k(T)^\perp$ for all $j\ge 0$. 
 By Lemma \ref{EigenInv}, $f\circ T^j \in G_k(T^\ell)^\perp$ for all $j\ge 0$. So, we can apply (\ref{Conl-2})
 to the system $(X, T^\ell)$, the function $f\circ T^j$ and $\ell^k \alpha$ (replacing $\ell \alpha$) in order  to get
\begin{equation*}
   \forall x \in X, \ \ \  \lim_{N\to \infty} \sup_{Q\in \mathbb{R}_{k-1}[t]} \left| \frac{1}{N}\sum_{n=0}^{N-1} e^{2\pi i(Q(n) +  n^k \ell^k \alpha)} f\circ T^j(T^{\ell n} x) \right|= 0
 \end{equation*}
 which is equivalent to 
 \begin{equation*}
   \forall x \in X, \ \ \  \lim_{N\to \infty} \sup_{Q\in \mathbb{R}_{k-1}[t]} \left| \frac{1}{N}\sum_{n=0}^{N-1} e^{2\pi i(Q(n\ell +j) +  (n\ell+j)^k \ell \alpha)} f(T^{n \ell  +j} x) \right|= 0.
 \end{equation*}This allows us to deduce (\ref{AA}) by taking average over $0\le j <\ell$.


\end{proof}

\section{Nilsystems}
Let $s\ge 1$ be an integer. Let $N$ be a $s$-step, simply connected nilpotent Lie group,  $\Gamma$  a discrete subgroup of $N$ such that $N/\Gamma$  is compact. The $s$-step nilpotence means that we have the following  lower central series
$$
   G_1 \vartriangleright G_2  \vartriangleright \cdots \vartriangleright G_s \vartriangleright G_{s+1}=\{e\}
$$
where   $G_{i+1} = [N,G_i]$ for
$i \ge 1$  with $G_0 = G_1 = N$. Recall that the commutator group $[N,G_i]$ is the group generated by
all commutators $hgh^{-1}g^{-1}$ with  $h \in N,   g \in G_i$. 
Then the quotient space $N/\Gamma$ is
called an $s$-step nilmanifold. Any $g\in N$   acts on $N/\Gamma$ by left multiplication $x \Gamma \mapsto gx \Gamma$.
This left translation will be denoted by $T_g: N/\Gamma \to N/\Gamma$, called a $s$-step nilsystem. 

A (basic) $s$-step nilsequence is a sequence of
the form $(f(T_g^n x))$ i.e. $(f(g^n x))$ where $x$ is a point of $N/\Gamma$ and $f: N/\Gamma\to \mathbb{C}$ a continuous function.

The additive group $\mathbb{R}^d$ is $1$-step nilpotent and the torus $\mathbb{T}^d: =\mathbb{R}^d/\mathbb{Z}^d$ is a $1$-step nilmanifold. 

\subsection{Fully oscillating nilsequences} The following is the restatement of Theorem B in Introduction. 

\begin{thm}  \label{Bb}
Let $G$ be a connected and simply connected nilpotent Lie group, $\Gamma$ a discrete cocompact subgroup of $G$ and $g\in G$. Let $X=G/\Gamma$ be the nilmanifold and let   $T: X \to X$ be defined by $x\Gamma \mapsto gx\Gamma$.
Suppose that $(X, T)$ is uniquely ergodic. Then for any $F \in C(X)$ such that $F\in G_1(T)^\perp$ and any $x \in G$, the sequence $F(g^n x \Gamma)$ is fully oscillating. 
\end{thm}

This is a direct consequence of Theorem \ref{OTWWT}, because there is no other quasi-eigenfunctions than eigenfunctions, which are all continuous, for any ergodic nilsystem
      associated to a connected and simply
      connected nilpotent Lie groups. 
In fact, if the system had higher order eigenfunctions, then it would have 
second order eigenfunctions.  These second order eigenfunctions  live on a 2-step factor which would also be associated to connected and
simply connected nilpotent Lie group, but which supports no second order eigenfunctions  other than true eigenfunctions.\medskip

It is the moment to give some comments. First recall the following theorem due to Lesigne. 

\begin{thm}[Lesigne \cite{L1989,L1991}] \label{Lesigne}Let $N/\Gamma$ be a nilmanifold and let $a\in N$. For any continuous function $f\in C(N/\Gamma)$ and any point $x \in N/\Gamma$, the following limit
exists
$$
         \lim_{N\to \infty} \frac{1}{N}\sum_{n=0}^{N-1} f(a^n x).
$$
\end{thm} 

The essential point of this theorem is the everywhere existence of the limit of the ergodic averages (with constant weights).  
The almost every convergence of the following multiple ergodic limit 
\begin{equation}\label{Nil-ErgodicMean}
        \lim_{N\to \infty} \frac{1}{N}\sum_{n=0}^{N-1} \prod_{k=1}^\ell f_k(a^{kn} x), \quad (f_1, \cdots, f_\ell \in L^\infty(m_{N/\Gamma})
\end{equation}
was proved by Lesigne \cite{L1989, L1991}. Under the assumption of ergodicity, an explicit formula for the limit was found by Lesigne \cite{L1991}
for $2$-step nilsystems, by Ziegler \cite{Ziegler2005} for all nilsystems. This formula was later generalized to the case where $n,2n, \cdots, \ell n$
are replaced by polynomials by Leibman \cite{Leibman2005}.

As pointed by B. Host \cite{Host2016} (personal communication),
it could be possible to deduce Theorem \ref{Bb}  from Theorem \ref{Lesigne}. The reason is as follows.
Let $U$ be a $d\times d$ unipotent
matrix with integer entries and $b \in \mathbb{T}^d$. Then the affine map
$S y = Uy +b$ 
defines an  affine $d$-step nilsystem.
If $P\in \mathbb{R}_d[t]$,
then the sequence $e^{2 \pi i P(n)}$ 
is produced  by an affine  $d$-step nilsystem, namely  there exists an affine $d$-step nilsystem $(\mathbb{T}^d, S)$ and a point $y_0\in \mathbb{T}^d$  such that
$e^{2 \pi i P(n)} = f(S^n y_0)$ 
 for every $n$.
Let $(N/\Gamma, T_g)$ be an ergodic nilsystem (hence minimal and uniquely
ergodic), a continuous function $F$ on $N/\Gamma$ and $x_0\in N/\Gamma$. 
Let $P, S, y_0$ and $f$ be as above. Then the sequence of general term $F(g^n x_0) e^{2\pi i P(n)}$
is produced by the nilsystem $(X\times Y, T\times S)$: 
$$
      F(g^n x_0)e^{2\pi i P(n)}  =  (F\otimes  f)(T_g \times S)^n(x_0,  y_0).
 $$
It follows from Theorem \ref{Lesigne} that the averages of this sequence
converge. More precisely, by Leibman \cite{Leibman2005} there exists a sub-nilmanifold $W$ of $X\times Y$  such that
\begin{equation}\label{LL}
    \lim_{N\to \infty}\frac{1}{N} \sum_{n=0}^{N-1}
     F(g^n x_0) e^{2\pi i P(n)} = \int_{W} F(x)f(y) dm_W(x, y)
 \end{equation}
 where $m_W$ is the Haar measure on $W$.
Note that $(W,  T\times S)$ is a joining of $(X, T_g)$ and $(Y, S)$.
However, to complete the argument, we need to prove that the integral on the RHS of (\ref{LL}) is equal to zero.

\medskip

\subsection{3-dimensional Heisenberg group}  Before proving Theorem C, we discuss a special Heisenberg group.
The 3-dimensional Heisenberg group
$$
       H:=\begin{pmatrix}   1  & \mathbb{R} &  \mathbb{R}\\
                                         &    1              &   \mathbb{R}\\
                                        &                     &  1
              \end{pmatrix}
$$
is a $2$-step simply connected nilpotent Lie group.  The group operation is the matrix multiplication. If we simply write $\langle x, y, z\rangle$ for an element of $H$, then the group  operation in $H$ is defined by
\begin{equation}\label{H-mutiplication}
  \langle a,b,c \rangle \langle x, y, z \rangle = \langle a+x, b+y, c+z + ay\rangle .
\end{equation}
If we take the subgroup
$$
       \Gamma:=\begin{pmatrix}   1  & \mathbb{Z} &  \mathbb{Z}\\
                                         &    1              &   \mathbb{Z}\\
                                        &                     &  1
              \end{pmatrix},
$$
we get a $2$-step nilmanifold $H/\Gamma$. Let $g= \langle \alpha_1,\alpha_2,\alpha_3 \rangle \in H$. Then 
$$
    T_g \langle x,y,z\rangle = \langle x+ \alpha_1,y + \alpha_2, z + \alpha_3 +  \alpha_1 y\rangle   \mod \Gamma.
$$
Take $\mathcal{F}= [0,1)\times [0,1) \times [0,1)$ as fundamental domain of $H/\Gamma$.
For $x =   \langle x_1, x_2, x_3\rangle \in N$, let 
$$
  \gamma_x =   \langle -[x_1], -[x_2], - [x_3 - x_1[x_2]]\rangle \in \Gamma.
$$
Then $\tau(x): =x \gamma_x \in \mathcal{F}$. Such a $\gamma_x$ is unique.  Actually we have
\begin{equation}\label{tau}
   \tau(x) =   \langle \{x_1\}, \{x_2\}, \{x_3 - x_1[x_2]\}\rangle .
\end{equation}

It can be  inductively proved that 
$$
   g^n =   \langle n\alpha_1, n \alpha_2, n \alpha_3 + C^2_n \alpha_1\alpha_2 \rangle.
$$
Then for any $x =   \langle x_1, x_2, x_3\rangle \in N$, we have
$$
   g^n x =   \langle n\alpha_1 +x_1, n \alpha_2 +x_2, n \alpha_3+x_3 + C_n^2\alpha_1\alpha_2 + n\alpha_1 x_2 \rangle. 
$$
Then for the map $T_g : N/\Gamma \to N/\Gamma$ with $N/\Gamma$ represented by the fundamental domain $\mathcal{F}$, by (\ref{tau}) we have
\begin{equation}\label{gnx}
  T_g^n x =  \langle n\alpha_1 +x_1, n \alpha_2 +x_2, n \alpha_3+x_3 + C_n^2\alpha_1\alpha_2 + n\alpha_1 x_2 - (n\alpha_1 +x_1)[n\alpha_2 +x_2] \rangle 
\end{equation}
where  the coordinates on the RHS are considered $\mod \Gamma$.
In particular,
\begin{equation}
  T_g^n 0 =  \langle \{n\alpha_1\}, \{n \alpha_2\}, \{n \alpha_3 + C_n^2\alpha_1\alpha_2 -  n\alpha_1 [n\alpha_2]\} \rangle.
\end{equation}
\medskip

Let us give here a direct proof of no second order quasi-functions in the case of Heisenberg ergodic translation. 
Let $F(x_1, x_2, x_3)$ be a second order quasi-eigenfunction, i.e. $T_gF = h F$ with $h$  an eigenfunction, which is of the form $ae^{2\pi i (kx+jy)}$ with $|a|=1$, $(k, j)\in \mathbb{Z}^2$. Since $F$ is orthogonal to all eigenfunctions,
$F$ is independent of $x_1$ and $x_2$. So $F(x_1, x_2, x_3) = f(x_3)$ for some function $f$. Then, using (\ref{tau}) and (\ref{gnx}), we can write the equation $T_gF = h F$  as
$$
    f(\alpha_3 + x_3 +\alpha_1 x_2 -[\alpha_1 +x_1](\alpha_2 +x_2)) = h(x_1, x_2) f(x_3).
$$
For almost all $(x_1, x_2)$ we develop the function of $x_3$ into Fourier series
$$
     \sum_n \widehat{f}(n) e^{2\pi i n(\alpha_3 +\alpha_1 x_2 -(\alpha_1 +x_1)[\alpha_2 +x_2])} e^{2\pi i n x_3}
        = h(x_1, x_2) \sum_n \widehat{f}(n) e^{2\pi i n x_3}
$$
So, by comparing the Fourier coefficients, we obtain that, for a fixed $n$,  either $\widehat{f}(n) =0$ or for almost all $(x_1, x_2)$
$$
    e^{2\pi i n(\alpha_3 +\alpha_1 x_2 -(\alpha_1 +x_1)[\alpha_2 +x_2])} 
        = h(x_1, x_2) .
$$
In other words, for   $0\le x_2<1-\alpha_1$ we have
$$e^{2\pi i n(\alpha_3 +\alpha_1 x_2)} 
        = h(x_1, x_2) $$  which is impossible;  and for   $1-\alpha_1\le x_2 <1$ we have
$$e^{2\pi i n(\alpha_3-\alpha_1 +\alpha_1 x_2 -x_1)}  = h(x_1, x_2),$$
which is impossible too. Thus $F$ must be constant and $h$ must be $1$.  
\medskip

Let us consider a Mal'cev basis of $H$, consisting of 
$$
       e_1=\begin{pmatrix}   1  &1 &  0\\
                                         &    1              &  0\\
                                        &                     &  1
              \end{pmatrix},\quad
               e_2=\begin{pmatrix}   1  &0 &  0\\
                                         &    1              &  1\\
                                        &                     &  1
              \end{pmatrix},
              \quad
               e_2=\begin{pmatrix}   1  & 0&  1\\
                                         &    1              &  0\\
                                        &                     &  1
              \end{pmatrix}.
$$
See \cite{Corwin-Greenleaf} for Mal'cev bases.
These elements determine three one-parameter subgroups $(e^t_i)_{t \in \mathbb{R}}$ ($i=1,2,3$): 
$$
       e_1^t=\begin{pmatrix}   1  &t &  0\\
                                         &    1              &  0\\
                                        &                     &  1
              \end{pmatrix},\quad
               e_2^{t}=\begin{pmatrix}   1  &0 &  0\\
                                         &    1              &  t\\
                                        &                     &  1
              \end{pmatrix},
              \quad
               e_2^t=\begin{pmatrix}   1  & 0&  t\\
                                         &    1              &  0\\
                                        &                     &  1
              \end{pmatrix}.
$$
Any element $g\in H$ has a unique representation as follows
$$
      g=e_1^{t_1}e_2^{t_2} e_3^{t_3} = \begin{pmatrix}   1  & t_1&  t_3 + t_1t_2\\
                                         &    1              &  t_2\\
                                        &                     &  1
                                          \end{pmatrix}.
$$
The triple $t_1, t_2, t_3$ will be denoted $\langle t_1, t_2, t_3\rangle_{\mbox{\tiny \rm II}}$, called the Mal'cev coordinates of second kind of $g$.
We write  $g = \phi_{\mbox{\tiny \rm II}}(t_1, t_2,t_3)$, or simply $g=\langle t_1, t_2, t_3\rangle_{\mbox{\tiny \rm II}}$. Notice that $ \phi_{\mbox{\tiny \rm II}}: \mathbb{R}^3 \to H$
is a diffeomorphism and that
$\Gamma =  \phi_{\mbox{\tiny \rm II}} (\mathbb{Z}^3)$. Also notice that
$$
    \langle t_1, t_2, t_3\rangle_{\mbox{\tiny \rm II}} = \langle t_1, t_2, t_3 + t_1t_2\rangle,\quad
     \langle a, b, c\rangle = \langle a, b, c  -ab \rangle_{\mbox{\tiny \rm II}}.
$$

The group law is expressed by the Mal'cev coordinates as follows
\begin{equation}\label{Mal2op}
     \langle t_1, t_2, t_3\rangle_{\mbox{\tiny \rm II}} * \langle s_1, s_2, s_3\rangle_{\mbox{\tiny \rm II}}
     = \langle t_1 +s_1, t_2+s_2, t_3+s_3 - t_2 s_1\rangle_{\mbox{\tiny \rm II}}.
\end{equation}
In fact
\begin{eqnarray*}
   && \langle t_1, t_2, t_3\rangle_{\mbox{\tiny \rm II}} * \langle s_1, s_2, s_3\rangle_{\mbox{\tiny \rm II}}\\
   & = & \langle t_1, t_2, t_3 + t_1 t_2\rangle \langle s_1, s_2, s_3 + s_1s_2\rangle\\
   & = & \langle t_1+s_1, t_2 +s_2, (t_3 + t_1 t_2)+(s_3 + s_1s_2) +t_1s_2\rangle \\
    & = & \langle t_1+s_1, t_2 +s_2, (t_3 + t_1 t_2)+(s_3 + s_1s_2) +t_1s_2 - (t_1+s_1)(t_2+s_2)\rangle_{\mbox{\tiny \rm II}}\\
    & = & \langle t_1+s_1, t_2 +s_2, t_3 + s_3 - t_2 s_1\rangle_{\mbox{\tiny \rm II}}
\end{eqnarray*}

Let $x =   \langle x_1, x_2, x_3\rangle_{\mbox{\tiny \rm II}}$. Let 
$$
  \gamma_x =   \langle -[x_1], -[x_2], - [x_3 + [x_1]x_2]\rangle_{\mbox{\tiny \rm II}} \in \Gamma.
$$
Then $\tau_2(x): =x \gamma_x \in \mathcal{F}$. We have
$$
   \tau(x) =   \langle \{x_1\}, \{x_2\}, \{x_3 + [x_1]x_2\}\rangle_{\mbox{\tiny \rm II}}
$$

Let $g =   \langle \alpha_1, \alpha_2, \alpha_3\rangle_{\mbox{\tiny \rm II}}$. Inductively we get
$$
   g^n =   \langle n\alpha_1, n \alpha_2, n \alpha_3 -C_n^2 \alpha_1\alpha_2 \rangle_{\mbox{\tiny \rm II}}
$$
then for any $x =   \langle x_1, x_2, x_3\rangle_{\mbox{\tiny \rm II}}$, we have
$$
   g^n x =   \langle n\alpha_1 +x_1, n \alpha_2 +x_2, n \alpha_3+x_3 - C_n^2\alpha_1\alpha_2 - n\alpha_2 x_1 \rangle_{\mbox{\tiny \rm II}}
$$
Then for $T_g : N/\Gamma \to N/\Gamma \to$ with $N/\Gamma$ represented by the fundamental domain $\mathcal{F}$,  $\mod 1$ we have
$$
  T_g^n x =  \langle n\alpha_1 +x_1, n \alpha_2 +x_2, n \alpha_3+x_3 - C_n^2\alpha_1\alpha_2 - n\alpha_2 x_1 - [n\alpha_1 +x_1](n\alpha_2 +x_2) \rangle_{\mbox{\tiny \rm II}}   
$$
In particular
\begin{equation}\label{TgII}
  T_g^n 0 =  \langle n\alpha_1, n \alpha_2, n \alpha_3 - C_n^2\alpha_1\alpha_2 -  [n\alpha_1] n\alpha_2 \rangle_{\mbox{\tiny \rm II}}     \mod 1.
\end{equation}




\medskip
\subsection{Proof of Theorem C}
Consider the $(2m+1)$-dimensional Heisenberg group $H_m$ ($m\ge 1$), which is the space $\mathbb{R}^{2m+1} = \mathbb{R}^m \times \mathbb{R}^m \times \mathbb{R}$
equipped with the group law defined by 
   \begin{equation}\label{Hn-mutiplication}
  \langle a,b,c \rangle \langle x, y, z \rangle = \langle a+x, b+y, c+z + B(a, y)\rangle 
\end{equation}
for $(a, b, c)\in  \mathbb{R}^m \times \mathbb{R}^m \times \mathbb{R}$ and $(x, y, z) \in  \mathbb{R}^m \times \mathbb{R}^m \times \mathbb{R}$, where 
$B:  \mathbb{R}^m \times \mathbb{R}^m \to \mathbb{R}$ is  the bilinear form
$$
    B(a, y) = \sum_{i=1}^m a_i y_i.
$$
Let $g= \langle \alpha,\beta,\gamma \rangle \in H_n$ with $\alpha = (\alpha_1, \cdots, \alpha_m)$, $\beta=(\beta_1, \cdots, \beta_m)$ and $\gamma \in \mathbb{R}$ (any choice for $\gamma$). We consider the translation $T_g$ defined by $g$:
\begin{equation}\label{H-mutiplication}
  T_g\langle x, y, z \rangle = \langle \alpha+x, \beta+y, \gamma +z + B(\alpha, y)\rangle .
\end{equation}
Take $\Gamma_m =\mathbb{Z}^{2m+1} $. 
For $\overline{x}=\langle x, y, z\rangle \in H_m$. Let 
$$
  \gamma_{\overline{x}} =   \langle -[x], -[y], - [z - B(x, [y])\rangle \in \Gamma_m.
$$
Then $\tau({\overline{x}}): ={\overline{x}} \gamma_{\overline{x}} \in \mathcal{F}_m: =[0,1)^{2m+1}$. Such a $\gamma_{\overline{x}}$ is unique.  We have
\begin{equation}\label{taum}
   \tau(\overline{x}) =   \langle \{x\}, \{y\}, \{z- B(x, [y])\}\rangle 
\end{equation}

We have
$$
   g^n =   \langle n\alpha, n \beta, n \gamma + C^2_n B(\alpha, \beta) \rangle.
$$
Then for any $\overline{x} =   \langle x, y, z\rangle \in N:=H_m/\Gamma_m$, we have
$$
   g^n \overline{x} =   \langle n\alpha +x, n \beta +y, n \gamma+z + C_n^2B(\alpha,\beta)+ n B(\alpha,  y) \rangle 
$$
Then for the map $T_g : H_m/\Gamma_m \to H_m/\Gamma_m$ with $H_m/\Gamma_m$ represented by the fundamental domain $\mathcal{F}_m$, by (\ref{taum}) we have
\begin{equation}
  T_g^n 0 =  \langle \{n\alpha\}, \{n \beta\}, \{n \gamma + C_n^2B(\alpha, \beta) -  B(n\alpha,  [n\beta]\}) \rangle.
\end{equation}
The condition on $\alpha$ and $\beta$ implies that $T_g$ is totally ergodic, by a theorem of Green \cite{L-Green} (see a simpler proof in \cite{Parry1970})
and that there is no quasi eigenfunctions other than true eigenfunctions.

Let $$
\omega_n =  n \gamma + C_n^2B(\alpha, \beta) -   B( n\alpha,  [n\beta]])\ \ \   (\!\!\!\!\! \mod 1). 
$$
Fix an integer $m \in \mathbb{Z} \setminus\{0\}$. We can apply Theorem \ref{Bb} to $F(x, y, z) =e^{2\pi i m z}$, which is orthogonal to all eigenfunctions, and  we get that
the sequence
$(e^{2\pi i m \omega_n})$ is orthogonal to all polynomial sequences $e^{2\pi i P(n)}$ with $P\in \mathbb{R}[t]$. In other words,  
$(e^{2\pi i m \omega_n})$ is fully oscillating. 
Since $Q(n) =  n \gamma + C_n^2B(\alpha, \beta)$ is a real polynomial of $n$,    so the sequence 
$(e^{- 2\pi i m B(n\alpha,  [n\beta])})$ then the sequence $(e^{ 2\pi i m B(n\alpha,  [n\beta])})$ is fully oscillating.

Let $\varphi \in C(\mathbb{T})$ with $\int \varphi(x)dx=0$. Now we claim that the sequence  $\varphi (B(n\alpha,  [n\beta]))$
is fully oscillating. In fact, for any $\epsilon>0$ there exists a trigonometric polynomial $g$ without the constant term such that 
$|\varphi(x)-g(x)|<\epsilon$. Then for any $P\in \mathbb{R}[t]$ we have
\begin{eqnarray*}
 &&\limsup_{N\to \infty}\left|\frac{1}{N}\sum_{n=0}^{N-1} e^{2\pi i P(n)} \varphi(B(n\alpha, [n\beta]))\right|\\
 &\le&  \limsup_{N\to \infty}\left|\frac{1}{N}\sum_{n=0}^{N-1} e^{2\pi i P(n)} g(B(n\alpha, [n\beta]))\right| + \epsilon.
\end{eqnarray*}
But  the last limit is equal to zero by   the full oscillation of $(e^{2\pi i mB(n\alpha, [n\beta]) })$ that we have already proved. 
\medskip

Similarly we can treat other generalized polynomial sequences by looking at different nilmanifolds. 

Bergelson \cite{Bergelson2017} pointed out to us that  Lemma 5.1 of Haland \cite{Haland1993}  
stated a result similar to Theorem C for polynomials of degree $2$.
Konieczy \cite{Konieczny} observed that it could be possible to give an alternative proof of Theorem C, replacing the application of Theorem B by the application of  the equidistribution developed by Leibman \cite{Leibman2012}.

\medskip
{\em Acknowledgement.}  I would like to thank V. Bergelson, B. Host, A. Leibman, M. Lemanczyk, B. Weiss, T. Ziegler  for valuable suggestions and informations.  This work is partially supported by NSFC No. 11471132 (China) and by Knuth and Alice Wallenberg Foundation (Sweden). 



\begin{thebibliography}{99}

\bibitem{Abramov}
L. M. Abramov, {\em Metric automorphisms with quasi-discrete spectrum}, Izv. Akad. Nauk. U. S. S. R., 26 (1962), 513-530.

\bibitem{Assani1993} I. Assani,
{\em Uniform Wiener-Wintner theorems for weakly mixing dynamical systems},  unpublished preprint, 1993. 

\bibitem{Assani} I. Assani,
{\em Wiener Wintner Ergodic theorems},  World Scientific Publishing Company, 2003. 

\bibitem{AGH1963}
L. Auslander, L. Green and F. Hahn, {\em Flows on homogeneous spaces},
 With the assistance of L. Markus and W. Massey, and an appendix by L. Greenberg. 
 Annals of Mathematics Studies, No. 53, Princeton University Press, Princeton, N.J., 1963, vii+107 pages.
 
 \bibitem{Bergelson2017} 
V. Bergelson,  Personal  communication.  
 
 \bibitem{BL2007}
 V. Bergelson and A. Leibman, {\em Distribution of values of bounded generalized polynomials} , Acta. Math. 198 (2007), no. 2, 155-230.
 

\bibitem{BL} A Bellow and V. Losert, {\em The weighted pointwise ergodic theorem and the individual ergodic theorem along subsequences},
Trans. Amer.Math. Soc., 288 (1985), 307-345.

\bibitem{Bourkaki}
N. Bourbaki, {\em Lie groups and Lie algebras.}, Chapters 1?3, Elements of Mathematics (Berlin), Springer-Verlag, Berlin, 1998, 
Translated from the French, Reprint of the 1989 English translation, xviii+450 pages.

                                                                                                                                                                                                                                          
 \bibitem{Bourgain1990} J. Bourgain, {\em 
Double recurrence and almost sure convergence},  J. Reine Angew. Math. 404 (1990), 140-161.

\bibitem{Bourgain13a} J. Bourgain, {\em On the correlation of the Moebius function with rank-one systems},
J. Anal. Math. 120 (2013), 105–130.

\bibitem{Bourgain13b} J. Bourgain, {\em M\"{o}bius-Walsh correlation bounds and an estimate of Mauduit
and Rivat}, J. Anal. Math. 119 (2013), 147-163.

\bibitem{BSZ} J. Bourgain, P. Sarnak, and Ziegler, {\em Disjointness of M\"obius from horocycle flows}. From Fourier
analysis and number theory to Radon transforms and geometry, 67-83, Dev. Math. 28, Springer,
New York, 2013, pp. 67-83.

 \bibitem{Corwin-Greenleaf}
L. J. Corwin and F. P. Greenleaf, {\em Representations of nilpotent Lie groups and their applications} Part I, Cambridge Studies in Advanced Mathematics,
 vol. 18, Cambridge University Press, Cambridge, 1990, Basic theory and examples, viii+269 pages.
 
  \bibitem{DK2015}
 T. Downarowicz and S. Kasjan, {\em Odometers and Toeplitz systems revisited in the context of
Sarnak’s conjecture}, Studia Math. 229 (2015), no. 1, 45-72.
 
 \bibitem{Eisner2015}  T. Eisner, {\em
 A polynomial version of Sarnak's conjecture}, C. R. Math. Acad. Sci. Paris 353 (2015), 569-572.

\bibitem{EZ_K} T. Eisner  and P. Zorin-Kranich, {\em Uniformity in the Wiener-Wintner theorem for nilsequences }, 
Discrete Contin. Dyn. Syst. 33 (2013) 3497-3516.


\bibitem{EALdlR14} E. H. El Abdalaoui, M. Lema\'nczyk, and T. de la Rue, {\em On spectral disjointness
of powers for rank-one transformations and M\"{o}bius orthogonality}, J. Funct.
Anal. 266 (2014), no. 1, 284-317.

\bibitem{EALdlR16} E. H. El Abdalaoui, M. Lema\'nczyk, and T. de la Rue, {\em Automorphisms with quasi-discrete
spectrum, multiplicative functions and average orthogonality along short intervals}, Int.Math.
Res. Not. IMRN (2016), to appear, arXiv:1507.04132v1.

\bibitem{EKLR17} El H. El Abdalaoui, J. Kulaga-Przymus, M. Lemanczyk and Th. De La Rue, {\em 
The Chowla and the Sarnak conjectures from ergodic theory point of view}, Discrete Contin. Dyn. Syst. 37
(2017), no.6, 2899-2944.







\bibitem{F} A. H. Fan, {\em Oscillating Sequences of higher orders and topological systems of
quasi-discrete spectrum}, preprint

\bibitem{F2} A. H. Fan, {\em Fully oscillating sequences and weighted multiple ergodic limit}, 
C. R. Acad. Sci. Paris, Ser.I,  Available online 31 July 2017.
arxiv.org:1705.01769.pdf
 
\bibitem{F3} A. H. Fan, {\em 
Weighted Birkhoff ergodic theorem with oscillating weights},  Ergod. Th. \& Dynam. Syst., to appear. arXiv:1705.02501


\bibitem{FJ}
A. H. Fan and Y. P. Jiang, {\em Oscillating sequences, minimal mean attractability
and minimal mean-Lyapunov-stability}, Erg. Th. Dynam. Syst.  doi:10.1017/etds.2016.121

\bibitem{FS}
A. H. Fan and D. Schneider, {\em Recurrence properties of sequences of integers}, Sci. China Math.  No. 3, Vol. 53 (2010), 641-656.

\bibitem{FKPLM15}  S Ferenzi, J. Kulaga-Przymus, M. Lema\'nczyk, and C Mauduit, {\em Substitutions
and M\"{o}bius disjointness}, preprint (2015).


\bibitem{Frantz2006}
N. Frantzikinakis, {\em Uniformity in the polynomial Wiener-Wintner theorem}, Ergod. Th. \& Dynam. Sys. 26 (2006), 1061-1071. 




\bibitem{Furstenberg1960}
H. Furstenberg, {\em Stationary processes and prediction theory}, Ann. Math. Studies 44, Princeton University Press, Princeton 1960. 


\bibitem{Furstenberg1961}
H. Furstenberg, {\em Strict ergodicity and transformations of the torus}, Amer. J. Math. 83 (1961), 573-601.

\bibitem{Furstenberg1990}
 H. Furstenberg, {\em Nonconventional ergodic averages}, in The legacy of John von Neumann (Hempstead, NY, 1988), Proc. Sympos. Pure Math., vol. 50, Amer. Math. Soc., Providence, RI, 1990, p. 43-56.

\bibitem{Green-Tao2008} B. Green and T. Tao, {\em Quadratic uniformity of the M\"{o}bius function},
Ann. Inst. Fourier, Grenoble 58, 6 (2008) 1863-1935.

\bibitem{GT2012} B. Green and T. Tao, {\em The M\"{o}bius function is strongly orthogonal to nilsequences},
Ann. of Math. (2) 175 (2012), no. 2, 541-566.


\bibitem{L-Green}
L. W. Green, {\em Spectra of nilflows}, Bull. Amer. Math. Soc. 67 (1961), 414-415.


\bibitem{HP0} F. Hahn and W. Parry,\,\emph{Minimal dynamical systems with quasi-discrete spectrum}. 
J. London Math. Soc., \textbf{40}\,(1965), 309-323.


\bibitem{Haland1993} 
I.  J.  H{a}land, {\em Uniform distribution of generalized polynomials}, Journal of Number Theory 45 (1993), 327-366.

\bibitem{Haland1994} 
I. J. Haland, {\em Uniform distribution of generalized polynomials of the product type}, Acta Arithmetica 67 (1994), 13-27.

\bibitem{Haland1995} 
 I. J. Haland and D. E. Knuth, {\em Polynomials involving the floor function}, Mathematica Scandinavica 76 (1995), 194-200.

\bibitem{HN} P. R. Halmos and J. von Neumann,\,\emph{Operator methods in classical mechanics. II}, 
Ann. of Math. (2) 43\,(1942), 332-350.

\bibitem{Host2016} 
B. Host,  Personal communication. 

\bibitem{HK2005} 
B. Host and B. Kra, {\em Nonconventional ergodic averages and nilmanifolds}, Ann.  Math. (2) 161 (2005), no. 1, 397-488. 

\bibitem{HK2009} B. Host and B. Kra, {\em Uniformity seminorms on $\ell^\infty$ and applications}, J. Anal. Math. 108 (2009), 219-276.

\bibitem{HLSY2017} W. Huang, Z. Lian, S. Shao, and X. Ye, Sequences from zero entropy noncommutative toral
automorphisms and Sarnak Conjecture, J. Differential Equations 263 (2017), no.1, 779-810.


\bibitem{HWY2017} W. Huang, Z. R. Wang and X. D.Ye, {\em Measure complexity and M\"{o}bius disjointness},  	arXiv:1707.06345.

\bibitem{HWZ2016} W. Huang, Z. Wang and G. Zhang, {\em M\"{o}bius disjointness for topological model of any ergodic
system with discrete spectrum}, preprint (2016), arXiv:1608.08289v2.

\bibitem{KPL2015} J. Kulaga-Przymus and M. Lema\'nczyk, {\em The M\"{o}bius function and continuous
extensions of rotations}, Monatsh. Math. 178 (2015), no. 4, 553-582.

\bibitem{Konieczny} 
J. Konieczny,  Personal  communication.  

\bibitem{Krengel}
U. Krengel, {\em Ergodic theorems}, Walter de Gruyter, Berlin, New York, 1982.


\bibitem{KPL15} J. Kulaga-Przymus and M. Lema\'nczyk, {\em The M\"{o}bius function and continuous
extensions of rotations}, Monatsh. Math. 178 (2015), no. 4, 553-582.


\bibitem{Leibman1998} A. Leibman, {\em Polynomial sequences in groups},  J. Algebra 201 (1998), 189-206.

 \bibitem{Leibman2005} A. Leibman, {\em Pointwise convergence of ergodic averages for polynomial sequences of rotations of a
nilmanifold},  Ergod. Th. \& Dynam. Sys. 25 (1) (2005), 201-213.

 \bibitem{Leibman2012} A. Leibman, {\em  A canonical form and the distribution of values of generalized polynomials},  Israel J. Math.,  188 (2012), 131-176.

\bibitem{L1989} E. Lesigne, {\em Th\'eor\`emes ergodiques pour une translation sur une nilvarie\'et\'e},
 Ergod. Th. \& Dynam. Sys. 9 (1) (1989), 115-126.
 
 \bibitem{L1990}
E. Lesigne, {\em Un th\'eor\`eme de disjonction de syst\`emes dynamiques et une g\'en\'eralisation du th\'eor\`eme ergidique de Wiener-Wintner},
Ergod. Th. \& Dynam. Syst., 10 (1990), 513-521.

 
\bibitem{L1991}  E. Lesigne, {\em Sur une nilvari\'et\'e, les parties minimales associ\'ees \`a une translation sont uniquement
ergodiques},  Ergod. Th. \& Dynam. Sys. 11(2) (1991), 379-391.


\bibitem{L1993}
E. Lesigne, {\em Spectre quasi-discret et th\'eor\`eme ergodique de Wiener-Wintner pour les polym\^omes},
Ergod. Th. \& Dynam. Syst., 13 (1993), 676-684.

\bibitem{Lenz2009}
D. Lenz, {\em Continuity of eigenfunctions of uniquely ergodic dynamical systems and intensity of Bragg peaks}, Comm. Math. Phys. 287 (2009), no. 1, 225-258.

\bibitem{LS} J. Y. Liu and P. Sarnak, {\em The M\"obius function and distal flows}, Duke Math. J. 164 (2015), no. 7, 1353-1399.

\bibitem{Malcev} 
 A. I. Mal'cev, {\em On a class of homogeneous spaces}, Izvestiya Akad. Nauk. SSSR. Ser. Mat. 13 (1949), 9-32.
 
 \bibitem{MR}  C. Mauduit and J. Rivat, {\em Sur un probl\`eme de Gelfond: la somme des chiffres des nombres premiers}. Ann. Math. (2) {\bf 171} (2010), 1591-1646.
 
 \bibitem{Parry1970} 
 W. Parry, {\em Dynamical systems on nilmanifolds}, Bull. London Math. Soc. 2 (1970), 37-40.
 

\bibitem{Robinson} E. A. Robinson, 
  {\em On uniform convergence in the Wiener-Wintner theorem},  J. London Math. Soc. (2) 49 (1994), 493-501.
  
  


  
  \bibitem{Sarnak}
P. Sarnak, {\em Three lectures on the M\"{o}bius function}, randomness and dynamics,
IAS Lecture Notes, 2009;\\
http://publications.ias.edu/sites/default/files/MobiusFunctionsLectures(2).pdf.
\bibitem{Sarnak2}
 P. Sarnak, {\em M\"{o}bius randomness and dynamics}, Not. S. Afr. Math. Soc. 43
(2012), 89-97.

 \bibitem{Shi} R. X. Shi, {\em Equivalent definitions of oscillating sequences of higher orders}, preprint. 
 
\bibitem{Wang2017}  Z. Wang, {\em M\"{o}bius disjointness for analytic skew products}, Invent. Math., Volume 209, Issue 1 (2017), 175–196.
 
 \bibitem{WW} N. Wiener and A. Wintner, {Harmonic analysis and ergodic theory}, Amer. J. Math., 63 (1941), 415-426.
 
 \bibitem{Veech} W.A. Veech, {\em M\"{o}bius orthogonality for generalized Morse-Kakutani flows}, Amer. J. Math.
(2016), to appear.
 
 \bibitem{Ziegler2005} T. Ziegler, {\em A non-conventional ergodic theorem for nilsystems},
 Ergod. Th. \& Dynam. Sys. (2005), 25, 1357-1370.

\end{thebibliography}

\end{document}